\documentclass{article}
\pdfoutput=1

\usepackage[a4paper]{geometry}

\usepackage{pdfsync}

\usepackage[hidelinks,hypertexnames=false,hyperindex=true,pdfpagelabels,linktoc=all]{hyperref}
\usepackage{latexsym,amsfonts,amsmath,amssymb}
\usepackage{graphicx}

\usepackage[mathscr]{euscript}

\usepackage[usenames,dvipsnames,svgnames]{xcolor}

\usepackage{amsthm}

\newtheorem{theorem}{Theorem}[section]
\newtheorem{lemma}[theorem]{Lemma}
\newtheorem{algorithm}[theorem]{Algorithm}

\theoremstyle{definition}

\numberwithin{equation}{section}

\newcommand{\satop}[2]{\stackrel{\scriptstyle{#1}}{\scriptstyle{#2}}}

\newcommand{\bsgamma}{{\boldsymbol{\gamma}}}

\newcommand{\bsnu}{{\boldsymbol{\nu}}}
\newcommand{\bsh}{{\boldsymbol{h}}}
\newcommand{\bsp}{{\boldsymbol{p}}}
\newcommand{\bsell}{{\boldsymbol{\ell}}}

\newcommand{\bsv}{{\boldsymbol{v}}}
\newcommand{\bsx}{{\boldsymbol{x}}}
\newcommand{\bsq}{{\boldsymbol{q}}}

\newcommand{\bsz}{{\boldsymbol{z}}}
\newcommand{\bsw}{{\boldsymbol{w}}}

\newcommand{\bszero}{{\boldsymbol{0}}}

\newcommand{\rd}{{\mathrm{d}}}
\newcommand{\ri}{{\mathrm{i}}}
\newcommand{\bbR}{\mathbb{R}}
\newcommand{\bbZ}{{\mathbb{Z}}}
\newcommand{\bbN}{{\mathbb{N}}}

\newcommand{\calA}{{\mathcal{A}}}

\newcommand{\calO}{{\mathcal{O}}}

\newcommand{\mask}[1]{}

\newcommand{\supp}{{\mathrm{supp}}}
\newcommand{\setu}{{\mathfrak{u}}}

\newcommand{\setw}{{\mathfrak{w}}}

\begin{document}

\title
{Fast Component-by-component Construction of \\ Lattice Algorithms for
Multivariate Approximation \\ with POD and SPOD weights
}

\author{
    Ronald Cools\footnote{Department of Computer Science, KU Leuven,
        Celestijnenlaan 200A, 3001 Leuven, Belgium,
        \texttt{(ronald.cools|dirk.nuyens)@cs.kuleuven.be}}
\and
    Frances Y.~Kuo\footnote{School of Mathematics and Statistics,
        University of New South Wales, Sydney NSW 2052, Australia,
        \texttt{(f.kuo|i.sloan)@unsw.edu.au}}
\and
    Dirk Nuyens\footnotemark[1]
\and
    Ian H.~Sloan\footnotemark[2]
}

\date{October 2019}
\maketitle

\begin{abstract}
In a recent paper by the same authors, we provided a theoretical
foundation for the component-by-component (CBC) construction of lattice
algorithms for multivariate $L_2$ approximation in the worst case setting,
for functions in a periodic space with general weight parameters. The
construction led to an error bound that achieves the best possible rate of
convergence for lattice algorithms. Previously available literature
covered only weights of a simple form commonly known as \emph{product
weights}. In this paper we address the computational aspect of the
construction. We develop \emph{fast} CBC construction of lattice
algorithms for special forms of weight parameters, including the so-called
\emph{POD weights} and \emph{SPOD weights} which arise from PDE
applications, making the lattice algorithms truly applicable in practice.
With $d$ denoting the dimension and $n$ the number of lattice points, we
show that the construction cost is $\calO(d\,n\log(n) + d^2\log(d)\,n)$
for POD weights, and $\calO(d\,n\log(n) + d^3\sigma^2\,n)$ for SPOD
weights of degree $\sigma\ge 2$. The resulting lattice generating vectors
can be used in other lattice-based approximation algorithms, including
kernel methods or splines.
\\[2mm]
\textbf{AMS Subject Classification: }41A10, 41A15, 65D30, 65D32, 65T40.
\end{abstract}
\section{Introduction}

In the paper \cite{CKNS-part1} we provided a theoretical foundation for
the component-by-component (CBC) construction of lattice algorithms for
multivariate $L_2$ approximation in the worst case setting, for functions
in a periodic space with general weight parameters. The construction led
to an error bound that achieves the best possible rate of convergence for
lattice algorithms. In this paper we address the computational aspect of
the construction. We develop \emph{fast} CBC construction of lattice
algorithms for special forms of the weight parameters, including the
so-called \emph{POD weights} and \emph{SPOD weights} which arise from PDE
applications, making the lattice algorithms truly applicable in practice.

The motivation for our work is the desire to use lattice algorithms (and
eventually kernel algorithms) to approximate the solution of a PDE with
random coefficients~\cite{CDS10}, as a function of the stochastic
variables. Previous related works
\cite{KSS12,DKLNS14,GKNSSS15,KN16,GKNSS18b,KKS} have been on approximating
the integral (expected value) of a linear functional of the PDE solution
with respect to the stochastic variables, rather than on directly
approximating the PDE solution itself. However, prior to our paper
\cite{CKNS-part1}, the existing literature on lattice algorithms for
approximation does not allow for weights of the POD or SPOD form. The
combination of the new theory in \cite{CKNS-part1} and the new algorithms
in this paper therefore provide the essential ingredients to apply lattice
algorithms to PDE applications.

We will provide some background in the introduction, assuming little prior
knowledge from the reader. A similar introduction can be found
in~\cite{CKNS-part1}, but here we focus more on the computational aspect.
Section~\ref{sec:form} provides the mathematical formulation of the
problem and reviews known results including those established
in~\cite{CKNS-part1}. In Section~\ref{sec:fast} we derive a new
formulation of the search criterion that enables the fast construction,
while in Sections~\ref{sec:product}--\ref{sec:SPOD} we develop fast CBC
constructions systematically for special forms of weights. In
Section~\ref{sec:num} we include numerical results for some artificial
choices of POD and SPOD weights. (More comprehensive experiments will
require us to choose weights based on the features of the given practical
problem and therefore go beyond the scope of this paper.)
Section~\ref{sec:conc} concludes the paper with our main theorem,
Theorem~\ref{thm:cost}, which summarizes the computational costs.

\subsection{Quasi-Monte Carlo methods and weighted spaces}

Quasi-Monte Carlo (QMC) methods are equal-weight cubature rules for
approximating high dimensional integrals. Reference books and surveys
include
\cite{Nie92,SJ94,Hic98b,HH02,CN08,Lem09,DP10,LM12,DKS13,LP14,Nuy14}. They
differ from the Monte Carlo methods in that the sample points are chosen
deterministically and more uniformly than random points, promising a
higher rate of convergence than the Monte Carlo root-mean-square error of
$\calO(n^{-1/2})$, with $n$ the number of sample points. There are two
main families of QMC point sets: \emph{digital nets} (\emph{and
sequences}) and \emph{lattice points}, both going back to Russian
number-theorists such as Sobol$'$, Hlawka and Korobov in the late 1950s.
Many QMC point sets and sequences, often collectively referred to as
\emph{low discrepancy sequences}, can achieve nearly first order
convergence rates for integration, while lattice points can achieve even
higher order convergence rates for smooth periodic integrands. However,
the implied constants in the big-$\calO$ bounds depend on the dimension
$d$, i.e., on the number of integration variables. For a long time it was
thought that QMC methods would not be effective in high dimensions,
because most theoretical error bounds for QMC methods contain a $\log(n)$
to a power depending on $d$. But this point of view has dramatically
changed in the last two decades.

The first breakthrough has been to analyze QMC methods in \emph{weighted
spaces} \cite{SW98,SW01,DSWW06}, following \emph{tractability} analysis
\cite{NW08,NW10,NW12}, to establish error bounds that are independent of
dimension. In effect through a choice of \emph{weight parameters} we
identify features of integrands that permit QMC methods to be effective in
very high dimensions. The second milestone has been the development of the
\emph{component-by-component} (CBC) constructions \cite{SR02,SKJ02b} and
\emph{fast CBC algorithms} \cite{NC06a,NC06b,NC06c,Nuy14}, which allow us
to obtain parameters for QMC point sets in thousands of dimensions and
with millions of points that are accompanied by a rigorous error analysis
\cite{Kuo03,DSWW06,DKS13}. The third landmark has been the invention of
\emph{higher order digital nets} for non-periodic integrands
\cite{Dic08,DP10,DKS13}.

Conceptually every function in $d$ dimensions can be expressed as a sum of
$2^d$ orthogonal terms \cite{KSWW10a} where each term depends only on a
subset $\setu$ of the $d$ variables, namely, $x_j$ for $j\in \setu
\subseteq \{1:d\} := \{1,\ldots,d\}$. Weight parameters allow us to
moderate the relative importance of these orthogonal terms. In the fullest
generality \cite{DSWW06} we assign a weight parameter $\gamma_\setu$ to
every subset of the integration variables $\bsx_\setu =
(x_j)_{j\in\setu}$. A small weight $\gamma_\setu$ then means that the
function depends weakly on $\bsx_\setu$. In this full generality there are
$2^d$ weight parameters to specify, which is infeasible in practice except
for very small $d$. So special forms of weights have been considered in
the literature:
\begin{itemize}
\item %
With \emph{product weights} \cite{SW98,SW01}, there is one weight
    parameter $\gamma_j>0$ associated with each coordinate direction
    $x_j$, and the weight for a subset of variables is taken to be the
    product
\[
     \gamma_\setu \,=\, \prod_{j\in\setu} \gamma_j.
\]%
So we have a sequence $\{\gamma_j\}_{j\ge 1}$ and we set
    $\gamma_\emptyset:=1$.
\item %
    With \emph{order dependent weights} \cite{DSWW06}, each
    $\gamma_\setu$ depends only on the cardinality of the set $\setu$,
\[
     \gamma_\setu \,=\, \Gamma_{|\setu|}.
\]
So they are described by a sequence $\{\Gamma_\ell\}_{\ell\ge 0}$,
     with $\gamma_\emptyset := \Gamma_0 :=1$. In addition, they are
     called \emph{finite order weights} of order $q$ if $\gamma_\setu$
    is zero for all subsets $\setu$ with cardinality greater than $q$.
\item %
Recent works on applying QMC for PDEs with random coefficients
    \cite{KSS12,GKNSSS15,KN16,GKNSS18b} have inspired a new form of
    weights called \emph{POD weights}, or \emph{product and order
    dependent weights}, which combine the features of product weights
    and order dependent weights,
\[
     \gamma_\setu \,=\, \Gamma_{|\setu|}\prod_{j\in\setu} \gamma_j.
\]
They are specified by two sequences $\{\gamma_j\}_{j\ge 1}$,
$\{\Gamma_\ell\}_{\ell\ge 0}$, with $\gamma_\emptyset := \Gamma_0
:=1$.
\item %
Further works on PDEs with random coefficients involving higher order
     QMC rules \cite{DKLNS14,KKS} have inspired a more complicated
     form of weights called \emph{SPOD weights}, or
     \emph{smoothness-driven product and order dependent weights},
     which involves an inner structure depending on a smoothness
     degree $\sigma\in\bbN$,
\[
  \gamma_\setu
  \,=\,
  \sum_{\bsnu_\setu \in \{1:\sigma\}^{|\setu|}}
  \Gamma_{|\bsnu_\setu|} \prod_{j \in \setu} \gamma_{j,\nu_j},
\]
where $|\bsnu_\setu| = \sum_{j\in\setu} \nu_j$, with $\gamma_\emptyset
:= \Gamma_0 :=1$. Note there is now a sequence
$\{\Gamma_\ell\}_{\ell>0}$, plus a sequence $\{\gamma_{j,\nu}\}_{j\ge
1}$ for each value of $\nu=1,\ldots,\sigma$.
\end{itemize}

\subsection{Construction of lattice rules for integration}

From here on we focus on the construction of \emph{lattice point sets} for
integrating and approximating periodic functions. Related results exist
for other QMC methods in non-periodic settings; the present work can also
be generalized.

An $n$-point (rank-$1$) lattice rule in $d$ dimensions is specified by an
integer vector $\bsz = (z_1,\ldots,z_d)$ called the \emph{generating
vector}. The resulting point set takes the form
\[
 \Big\{\Big\{\frac{k\bsz}{n}\Big\} \,:\, k\in\bbZ_n\Big\},
\]
where $\bbZ_n := \{0,1,\ldots,n-1\}$, and the inner pair of braces
indicates that we take the fractional part of each component in the
vector. The components of $\bsz$ can be restricted to the range
$\{1,\ldots,n-1\}$, so altogether there are $(n-1)^d$ possible choices for
the generating vector. If an error criterion for the lattice rule can be
evaluated in $\kappa(d,n)$ operations, then it would require
$\calO(n^d\,\kappa(d,n))$ operations to go through all choices to find one
with the smallest error, which is impossible to do when $d$ is large even
if $\kappa(d,n)=\calO(1)$. A CBC construction chooses the components of
the generating vector one at a time, with the previously chosen components
held fixed:
\begin{enumerate}
 \item Set $z_1=1$.
 \item With $z_1$ held fixed, choose $z_2\in \{1,\ldots,n-1\}$ to
     minimize the error criterion in $2$ dimensions.
 \item With $z_1,z_2$ held fixed, choose $z_3\in \{1,\ldots,n-1\}$ to
     minimize the error criterion in $3$ dimensions.
\item With $z_1,z_2,z_3$ held fixed, choose $z_4\in \{1,\ldots,n-1\}$
    to minimize the error criterion in $4$ dimensions.
\item [$\vdots$\,\,\,]
\end{enumerate}
In comparison with the cost of an exhaustive search above, a naive
implementation of the CBC construction requires only
$\calO(d\,n\,\kappa(d,n))$ operations.

For periodic integrands in the Hilbert space whose squared Fourier
coefficient decay at the rate of $\alpha>1$ (corresponding roughly to
$\alpha/2$ available mixed derivatives), it is known that lattice
generating vectors can be obtained by the CBC construction to achieve the
optimal convergence rate of $\calO(n^{-\alpha/2+\delta})$, $\delta>0$,
where the implied constant is independent of $d$ provided that the
(general) weights satisfy a certain summability condition \cite{DSWW06}.

For lattice rules in the periodic setting with product weights, the main
term in the error criterion takes the form \cite{SW01}
\[
  \sum_{k\in\bbZ_n} \prod_{j=1}^d \big(1 + \gamma_j\, \omega(z_j,k)\big),
\]
which can be computed in $\kappa(d,n)=\calO(d\,n)$ operations, so a naive
implementation of the CBC construction requires $\calO(d^2n^2)$
operations. This can be reduced to $\calO(d\,n^2)$ operations by storing
the products during the search. This can be further reduced to
$\calO(d\,n\log(n))$ operations by recognizing that the search involves a
matrix-vector product where the matrix
$[\omega(z,k)]_{z\in\{1,\ldots,n-1\},k\in\bbZ_n}$ can be turned into a
\emph{circulant} matrix, since $\omega(z,k)$ depends only on the value of
$(kz\bmod n)$, so that the fast Fourier transform (FFT) can be used
to speed up the computation \cite{NC06a,NC06b,NC06c,Nuy14}.

With general weights, the main term in the error criterion takes the form
\cite{DSWW06}
\[
  \sum_{k\in\bbZ_n} \sum_{\setu\subseteq\{1:d\}} \gamma_\setu
  \prod_{j\in\setu} \omega(z_j,k),
\]
which requires $\kappa(d,n)=\calO(2^d n)$ operations to evaluate, making
the CBC construction impossible. With order dependent weights
$\gamma_\setu = \Gamma_{|\setu|}$, this main term can be written as
\[
  \sum_{k\in\bbZ_n} \sum_{\ell=0}^d \Gamma_\ell
  \underbrace{\sum_{\setu\subseteq\{1:d\},\,|\setu|=\ell\;} \prod_{j\in\setu} \omega(z_j,k)}_{=:\, P_{d,\ell}(k)},
\]
where the quantities $P_{d,\ell}(k)$ can be stored and computed
recursively. This yields a fast CBC construction with cost
$\calO(d\,n\log(n) + d^2 n)$, where the second term arises due to the need
to update $P_{d,\ell}(k)$ \cite{CKN06}. The algorithm and cost for POD
weights is essentially the same as for order dependent weights
\cite{KSS11}. The algorithm for SPOD weights is more complicated but makes
use of similar ideas and has a cost of $\calO(d\,n\log(n) + d^2
\sigma^2\,n)$ \cite{KKS}.

\subsection{Construction of lattice algorithms for approximation}

Lattice point sets can be used to approximate a periodic function by first
truncating the Fourier series expansion to a finite index set, and then
approximating those Fourier coefficients (which are integrals of the
function against each basis function) by lattice rules. We refer to this
method of approximation as \emph{lattice algorithms}. Existing literature
on lattice-based approximation algorithms has been for the unweighted
setting or for product weights
\cite{LH03,KSW06,ZLH06,KSW08,ZKH09,Kam13,KPV15,SNC16,CKNS16,PV16,BKUV17,KMNN}.

The optimal algorithm for (worst case) $L_2$ approximation based on the
class of \emph{arbitrary linear information} (implying that all Fourier
coefficients can be obtained exactly) can achieve the convergence rate
$\calO(n^{-\alpha/2+\delta})$, $\delta>0$, same as for integration, see
\cite{NSW04}. However, if we restrict to the class of \emph{standard
information} where only function values are available, then it has been an
open problem whether the same rate can be achieved with no dependence of
the error bound on the dimension $d$. A general (non-constructive) result
in \cite{KWW09a} yields the convergence rate
$\calO(n^{-(\alpha/2)[1/(1+1/\alpha)]+\delta})$, $\delta>0$. A very recent
manuscript \cite{KUnew} appears to have solved this open problem.

For algorithms that use function values at lattice points, it has been
proved in~\cite{BKUV17} that the best possible convergence rate is
$\calO(n^{-\alpha/4+\delta})$, $\delta>0$. Hence, unfortunately, lattice
algorithms are not optimal. However, they do have a number of advantages,
including simplicity and efficiency, and therefore can still be
competitive. In~\cite{CKNS-part1} we proved that a lattice generating
vector can be obtained by a CBC algorithm for general weights to achieve
this best possible rate, see Theorem~\ref{thm:final} below.

The fast CBC construction of lattice algorithms for approximation with
non-product weights is much harder than for integration because the error
criterion is rather complicated. This is precisely the goal of this paper.
We show that the overall cost in obtaining a suitable lattice
generating vector is
\begin{align*}
  &\calO\big(d\,n\log(n)\big)  && \mbox{for product weights}, \\
  &\calO\big(d\,n\log(n) + d^2\log(d)\,n\big)  && \mbox{for order dependent weights and POD weights}, \\
  &\calO\big(d\,n\log(n) + d^3\sigma^2\,n\big) && \mbox{for SPOD weights with degree $\sigma \ge 2$},
\end{align*}
plus storage cost as well as pre-computation cost for POD and SPOD
weights, see Theorem~\ref{thm:cost}.

The essential ingredient in managing the computational cost for
non-product weights is to recognize that there are multiple matrix-vector
products involving \emph{Hankel} matrices (i.e., all anti-diagonals are
constant) and therefore the usual $\calO(d^2)$ complexity can be reduced
to $\calO(d\log(d))$ using FFT. This reduction is enough to bring the cost
down to nearly quadratic in $d$ for order dependent weights and POD
weights. Unfortunately, for SPOD weights there are other difficulties
which meant that the best we can do is cubic in $d$. We remark again that,
without special structure of the weights, the computational cost would be
exponentially high in $d$.

In the application of QMC methods to PDE problems, the weights are
typically chosen to minimize (or at least make small) the cubature error
bound, aiming at obtaining the best possible convergence rate while
keeping the error bound independent of the number of stochastic variables
\cite{KSS12,DKLNS14,GKNSSS15,KN16,GKNSS18b,KKS}. It is often the case that
the best theoretical convergence rate can only be obtained by choosing
weights of a more complicated form; this is how POD weights and SPOD
weights arose. For the integration problem, there is no essential
difference between the construction of lattice generating vectors with POD
or SPOD weights \cite{KKS}, but for the approximation problem SPOD weights
are more costly than POD weights as stated above. Thus it is then a
potential trade-off between the cost for the CBC construction and the
theoretical rate of convergence. One may argue that the CBC construction
cost should be considered an offline cost in the PDE application and it is
worth investing in SPOD weights so that the best possible convergence rate
is guaranteed, since every lattice point ultimately involves one
complicated PDE solve.

\section{Problem formulation and review of known results} \label{sec:form}

\subsection{Lattice rules and lattice algorithms}

We consider one-periodic real-valued $L_2$ functions defined on $[0,1]^d$
with absolutely convergent Fourier series
\begin{align*}
  f(\bsx) \,=\, \sum_{\bsh\in\bbZ^d} \hat{f}_\bsh\,e^{2\pi\ri\bsh\cdot\bsx},
  \quad\mbox{with}\quad
  \hat{f}_\bsh \,:=\, \int_{[0,1]^d} f(\bsx)\,e^{-2\pi\ri\bsh\cdot\bsx}\,\rd\bsx.
\end{align*}
where $\hat{f}_\bsh$ are the Fourier coefficients and $\bsh\cdot\bsx =
h_1x_1 + \cdots + h_dx_d$ denotes the usual dot product.

A (rank-$1$) \emph{lattice rule} \cite{SJ94} with $n$ points and
generating vector $\bsz\in \{1,\ldots,n-1\}^d$ approximates the
integral of $f$ by
\begin{align*}
  \int_{[0,1]^d} f(\bsx)\,\rd\bsx
  \quad\approx\quad
  \frac{1}{n} \sum_{k\in\bbZ_n} f\Big(\Big\{\frac{k\bsz}{n}\Big\}\Big),
\end{align*}
where the braces around a vector indicate that we take the fractional part
of each component in the vector.

A \emph{lattice algorithm} \cite{KSW06} with $n$ points and generating
vector $\bsz\in \{1,\ldots,n-1\}^d$, together with an index set
$\calA_d\subset \bbZ^d$, approximates the function $f$ by first truncating
the Fourier series to the finite index set $\calA_d$ and then
approximating the remaining Fourier coefficients by the lattice cubature
rule:
\begin{align} \label{eq:Af}
  A(f)(\bsx) \,:=\, \sum_{\bsh\in\calA_d} \hat{f}_\bsh^a \,e^{2\pi\ri\bsh\cdot\bsx},
  \quad\mbox{with}\quad
  \hat{f}_\bsh^a \,:=\, \frac{1}{n} \sum_{k\in\bbZ_n} f\Big(\Big\{\frac{k\bsz}{n}\Big\}\Big)\,e^{-2\pi\ri k\bsh\cdot\bsz/n}.
\end{align}

\subsection{Function space setting with general weights}

For $\alpha>1$ and nonnegative weight parameters
$\bsgamma=\{\gamma_\setu\}$, we consider the Hilbert space $H_d$ of
one-periodic real-valued $L_2$ functions defined on $[0,1]^d$ with
absolutely convergent Fourier series, with norm defined by
\begin{align*}
  \|f\|_d^2 \,:=\, \sum_{\bsh\in\bbZ^d} \big|\hat{f}_\bsh\big|^2\, r(\bsh),
  \quad\mbox{with}\quad
  r(\bsh) \,:=\, \frac{1}{\gamma_{\supp(\bsh)}}\,\prod_{j\in \supp(\bsh)} |h_j|^\alpha,
\end{align*}
where $\supp(\bsh) := \{1\le j\le d : h_j \ne 0\}$. The parameter $\alpha$
characterizes the rate of decay of the squared Fourier coefficients, so it
is a smoothness parameter. Taking $\gamma_\emptyset :=1$ ensures that the
norm of a constant function in $H_d$ matches its $L_2$ norm.

Some authors refer to this as the \emph{weighted Korobov space}, see
\cite{SW01} for product weights and \cite{DSWW06} for general weights,
while others call this a weighted variant of the \emph{periodic Sobolev
space with dominating mixed smoothness} \cite{BKUV17}.

When $\alpha\geq 2$ is an even integer, it can be shown that
\begin{align*}
 \|f\|_d^2 \,=\,
 \sum_{\setu\subseteq\{1:d\}}\frac{1}{(2\pi)^{\alpha|\setu|}}\frac{1}{\gamma_{\setu}}
 \int_{[0,1]^{|\setu|}}\!\!
 \bigg(\int_{[0,1]^{d-|\setu|}}\bigg(\prod_{j\in\setu}\frac{\partial}{\partial x_j}\bigg)^{\alpha/2}f(\bsx)
 \,\rd\bsx_{\{1:d\}\setminus\setu}\bigg)^2
 \rd\bsx_{\setu}.
\end{align*}
So $f$ has mixed partial derivatives of order $\alpha/2$ in each variable.
Here $\bsx_\setu = (x_j)_{j\in\setu}$.

\subsection{Approximation}

For the approximation problem we can follow \cite{KSW06,KSW08} to define
the index set $\calA_d$ with some parameter $M>0$ by
\begin{align} \label{eq:AdM}
  \calA_d(M) \,:=\, \big\{\bsh\in\bbZ^d : r(\bsh) \le M \big\},
\end{align}
with the difference being that here we have general weights determining
the values of $r(\bsh)$, while \cite{KSW06,KSW08} considered only product
weights. From \cite{KSW06,CKNS-part1} we have the \emph{worst case $L_2$
approximation} error bound
\begin{align*} 
  e^{\rm wor\mbox{-}app}_{n,d,M}(\bsz)
  \,:=\,& \sup_{f\in H_d,\,\|f\|_d\le 1} \|f - A(f)\|_{L_2} \nonumber\\
  \,\le\,& \bigg(\frac{1}{M} + E_d(\bsz)\bigg)^{1/2}
  \,\le\, \bigg(\frac{1}{M} + M\,S_d(\bsz)\bigg)^{1/2},
\end{align*}
with (in the last step using $r(\bsh)\le M$ for $\bsh\in\calA_d(M)$)
\begin{align*}
  E_d(\bsz)\,:=\, \sum_{\bsh\in\calA_d(M)} \sum_{\satop{\bsell\in\bbZ^d\setminus\{\bszero\}}{\bsell\cdot\bsz\equiv_n 0}}
  \frac{1}{r(\bsh+\bsell)}
  \quad\mbox{and}\quad
  S_d(\bsz)\,:=\, \sum_{\bsh\in\bbZ^d} \frac{1}{r(\bsh)}
  \sum_{\satop{\bsell\in\bbZ^d\setminus\{\bszero\}}{\bsell\cdot\bsz\equiv_n 0}} \frac{1}{r(\bsh+\bsell)}.
\end{align*}
The quantity $E_d(\bsz)$ was analyzed in \cite{KSW06,KSW08}, while a
variant of $S_d(\bsz)$ first appeared in the context of a
Lattice-Nystr\"om method for Fredholm integral equations of the second
kind \cite{DKKS07}. The advantage of working with $S_d(\bsz)$ instead of
$E_d(\bsz)$ is that there is no dependence on the index set $\calA_d(M)$.
This leads to an easier error analysis and a lower cost in finding
suitable generating vectors. The initial approximation error is given by
$e^{\rm wor\mbox{-}app}_{0,d} := \sup_{f\in H_d,\,\|f\|_d\le 1}
\|f\|_{L_2} = \max_{\setu\subseteq\{1:d\}} \gamma_\setu^{1/2}$.

\subsection{Collection of results from \cite{CKNS-part1}} \label{sec:paper1}

We proved in \cite{CKNS-part1} that a generating vector~$\bsz$ can be
constructed by a CBC algorithm based on $S_d(\bsz)$ with general weights
as the search criterion, so that the worst case $L_2$ approximation error
achieves the best possible rate for lattice algorithms. Our goal in
this paper is to develop fast CBC algorithms for special forms of
weights. Here we include some necessary results from \cite{CKNS-part1}.

The CBC algorithm works with a \emph{dimension-wise decomposition} of the
error criterion $S_d(\bsz)$ as shown in \eqref{eq:Sd-decomp} below.
Compared with most CBC algorithms, the difficulty for the error analysis
in \cite{CKNS-part1}, as well as the construction here, is that each step
relies on the entire weight sequence, i.e., ``future'' weights come into
play as can be seen from the expression \eqref{eq:Tds}. Thus the target
final dimension $d$ must be fixed at the start of the CBC algorithm, and
the resulting lattice generating vector is not extensible in $d$. Similar
strategies have been used previously in \cite{NSW17,ELN18}.

\begin{lemma}
Let $d\ge 1$ be fixed and a sequence of weights
$\{\gamma_\setu\}_{\setu\subseteq\{1:d\}}$ be given. We can write
\begin{align} \label{eq:Sd-decomp}
  S_d(\bsz)
  \,=\, \sum_{s=1}^d T_{d,s} \big(z_1,\ldots,z_s\big),
\end{align}
where, for each $s=1,2,\ldots,d$,
\begin{align} \label{eq:Tds}
  T_{d,s} \big(z_1,\ldots,z_s\big)
  \,:=\, \sum_{\setw\subseteq\{s+1:d\}}
  [2\zeta(2\alpha)]^{|\setw|}\, \theta_s \big(z_1,\ldots,z_s; \{\gamma_{\setu\cup\setw}\}_{\setu\subseteq\{1:s\}}\big),
\end{align}
\begin{align} \label{eq:theta}
  \theta_s \big(z_1,\ldots,z_s;\{\beta_\setu\}_{\setu\subseteq\{1:s\}}\big)
  \,:=\, \sum_{\bsh\in\bbZ^s} \sum_{\satop{\bsell\in\bbZ^s,\;\ell_s\ne 0}{\bsell\cdot (z_1,\ldots,z_s)\equiv_n 0}}
  \frac{\beta_{\supp(\bsh)}}{r'(\bsh)}
  \frac{\beta_{\supp(\bsh+\bsell)}}{r'(\bsh+\bsell)},
\end{align}
with $r'(\bsh) := \prod_{j\in\supp(\bsh)} |h_j|^\alpha$.
\end{lemma}

\begin{algorithm} \label{alg}
Given $n\ge 2$, a fixed $d\ge 1$, and a sequence of weights
$\{\gamma_\setu\}_{\setu\subseteq\{1:d\}}$, the generating vector $\bsz^*
= (z_1^*,\ldots,z_d^*)$ is constructed as follows: for each $s = 1,
\ldots,d$, with $z_1^*,\ldots,z_{s-1}^*$ fixed, choose $z_s \in
\{1,\ldots,n-1\}$ to minimize the quantity $T_{d,s}
\big(z_1^*,\ldots,z_{s-1}^*,z_s\big)$ given by \eqref{eq:Tds}.
\end{algorithm}

\begin{theorem} \label{thm:main}
Let $n$ be prime. For fixed $d\ge 1$ and a given sequence of weights
$\{\gamma_\setu\}_{\setu\subseteq\{1:d\}}$, a generating vector $\bsz$
obtained from the CBC construction following Algorithm~\ref{alg} satisfies
for all $\lambda\in (\tfrac{1}{\alpha},1]$,
\begin{align} \label{eq:final}
  S_d(\bsz) \,\le\,
  \bigg[ \frac{\tau}{n} \bigg(
  \sum_{\emptyset\ne\setu\subseteq\{1:d\}} |\setu|\,\gamma_{\setu}^\lambda\, [2\zeta(\alpha\lambda)]^{|\setu|}\bigg)
  \bigg(\sum_{\setu\subseteq\{1:d\}} \gamma_\setu^\lambda\, [2\zeta(\alpha\lambda)]^{|\setu|}\bigg) \bigg]^{1/\lambda},
\end{align}
where $\tau := \max(6,2.5+2^{2\alpha\lambda+1})$. Furthermore, if the
weights are such that there exists a constant $\xi\ge 1$
\textnormal{(}which may depend on $\lambda$\textnormal{)} such that
\begin{align} \label{eq:decay}
  \gamma_{\setu\cup\setw}^\lambda \le \xi\,\frac{\gamma_\setu^\lambda}{[2\zeta(\alpha\lambda)]^{|\setw|}}
  \quad\mbox{for all}\quad \setu\subseteq\{1:s\},\; \setw\subseteq\{s+1:d\},\; s\ge 1,\; d\ge 1,
\end{align}
then \eqref{eq:final} holds with $\tau$ replaced by $\tau\,\xi$ and with
the $|\setu|$ factor inside the first sum replaced by $1$.
\end{theorem}

\begin{theorem} \label{thm:final} Given $d\ge 1$, $\alpha>1$ and
weights $\{\gamma_\setu\}_{\setu\subset\bbN}$, let $n$ be prime and $M>0$.
The lattice algorithm \eqref{eq:Af}, with index set \eqref{eq:AdM} and
generating vector $\bsz$ obtained from the CBC construction following
Algorithm~\ref{alg}, satisfies for all $\lambda\in (\frac{1}{\alpha},1]$,
\begin{align*}
  e^{\rm wor\mbox{-}app}_{n,d,M}(\bsz)
  &\,\le\, \bigg(\frac{1}{M} + M\, S_d(\bsz) \bigg)^{1/2} \\
  &\,\le\, \Bigg(\frac{1}{M} + M \bigg[\frac{\tau}{n} \bigg(
  \sum_{\emptyset\ne\setu\subseteq\{1:d\}} |\setu|\,\gamma_{\setu}^\lambda\, [2\zeta(\alpha\lambda)]^{|\setu|}\bigg)
  \bigg(\sum_{\setu\subseteq\{1:d\}} \gamma_\setu^\lambda\, [2\zeta(\alpha\lambda)]^{|\setu|}\bigg)\bigg]^{1/\lambda}
 \Bigg)^{1/2},
\end{align*}
where $\tau = \max(6,2.5+2^{2\alpha\lambda+1})$.
Taking $M = n^{1/(2\lambda)}$, we obtain a simplified upper bound
\begin{align*}
  e^{\rm wor\mbox{-}app}_{n,d,M}(\bsz)
  \,\le\, \frac{\sqrt{2}\,\tau^{1/(2\lambda)}}{n^{1/(4\lambda)}} \bigg(
  \sum_{\setu\subseteq\{1:d\}} \max(|\setu|,1)\,\gamma_{\setu}^\lambda\, [2\zeta(\alpha\lambda)]^{|\setu|}\bigg)^{1/\lambda}.
\end{align*}
Hence
\[
  e^{\rm wor\mbox{-}app}_{n,d,M}(\bsz) \,=\, \calO(n^{-\alpha/4 + \delta}), \quad\delta>0,
\]
where the implied constant is independent of $d$ provided that
\begin{equation} \label{eq:condition}
  \sum_{\setu\subset\bbN,\,|\setu|<\infty} \max(|\setu|,1)\,\gamma_{\setu}^{\frac{1}{\alpha-4\delta}}\,
  [2\zeta\big(\tfrac{\alpha}{\alpha-4\delta}\big)]^{|\setu|}
  \,<\, \infty.
\end{equation}
If the weights satisfy \eqref{eq:decay} for some $\xi\ge 1$ then the
$|\setu|$ and $\max(|\setu|,1)$ factors inside the sums can be replaced by
$1$ as long as $\tau$ is replaced by $\tau\,\xi$.
\end{theorem}

We can apply the bound $\max(|\setu|,1)\le (e^{1/e})^{|\setu|}$ in
\eqref{eq:condition} to obtain a sufficient condition
$\sum_{\setu\subset\bbN,\,|\setu|<\infty}
\gamma_{\setu}^{\frac{1}{\alpha-4\delta}}\,
  [2e^{1/e}\zeta\big(\tfrac{\alpha}{\alpha-4\delta}\big)]^{|\setu|}
  \,<\, \infty$.

\section{New formulation of the search criterion} \label{sec:fast}

To be able to evaluate efficiently the quantity $T_{d,s}(z_1,\ldots,z_s)$
in \eqref{eq:Tds} which is needed in Algorithm~\ref{alg}, we proceed to
derive an alternative formulation which allows us to carry out the search
using two matrix-vector multiplications. Note that we do not require $n$
to be prime in Algorithm~\ref{alg} nor any of the subsequent derivations
in this paper. (Restricting $n$ to primes is used to simplify the error
analysis in \cite{CKNS-part1}; it should be possible to generalize the
results to composite $n$ with a more technical proof and modified
constants.)

\begin{lemma} \label{lem:semi}
We can rewrite the search criterion \eqref{eq:Tds} as
\begin{align*}
  T_{d,s}(z_1,\ldots,z_s)
  \,=\, \frac{1}{n} \sum_{k\in\bbZ_n} \psi(z_s,k)\,V_{d,s}(k) + \frac{2}{n} \sum_{k\in\bbZ_n} \omega(z_s,k)\,W_{d,s}(k),
\end{align*}
where, for $z\in \{1,\ldots,n-1\}$ and $k\in\bbZ_n$,
\begin{align} \label{eq:omega}
 \omega(z,k) &\,:=\,
 \sum_{h\in\bbZ\setminus\{0\}} \frac{e^{2\pi\ri k h z/n}}{|h|^\alpha},\qquad
 \psi(z,k) \,:=\, [\omega(z,k)]^2 - 2\zeta(2\alpha),
\end{align}
and
\begin{align*}
  V_{d,s}(k)
  &\,:=\,
  \sum_{\setw\subseteq\{s+1:d\}} [2\zeta(2\alpha)]^{|\setw|}\,
  \bigg(\sum_{\setu\subseteq\{1:s-1\}} \gamma_{\setu\cup\{s\}\cup\setw} \prod_{j\in\setu} \omega(z_j,k)\bigg)^2, \\
  W_{d,s}(k)
  &\,:=\, \sum_{\setw\subseteq\{s+1:d\}} [2\zeta(2\alpha)]^{|\setw|}
  \bigg(\!\sum_{\setu\subseteq\{1:s-1\}} \!\!\!\!\!\! \gamma_{\setu\cup\{s\}\cup\setw} \prod_{j\in\setu} \omega(z_j,k)\bigg)
  \bigg(\!\sum_{\setu\subseteq\{1:s-1\}} \!\!\!\!\!\!\gamma_{\setu\cup\setw} \prod_{j\in\setu} \omega(z_j,k)\bigg).
\end{align*}
Note that both $V_{d,s}(k)$ and $W_{d,s}(k)$ depend on
$z_1,\ldots,z_{s-1}$.
\end{lemma}

\begin{proof}
With the substitution $\bsq = \bsh + \bsell$ and the abbreviation $\bsz =
(z_1,\ldots,z_s)$, we can rewrite \eqref{eq:theta} as
\begin{align*}
  \theta_s(z_1,\ldots,z_s; \{\beta_\setu\}_{\setu\subseteq\{1:s\}})
  &\,=\,
  \sum_{\bsh \in \bbZ^s} \sum_{\substack{\bsq \in \bbZ^s,\, q_s \ne h_s \\ (\bsq-\bsh)\cdot\bsz \equiv_n 0}}
  \frac{\beta_{\supp(\bsh)}}{r'(\bsh)} \frac{\beta_{\supp(\bsq)}}{r'(\bsq)} \\
  &\,=\, \frac{1}{n} \sum_{k\in\bbZ_n}
  \sum_{\bsh \in \bbZ^s} \sum_{\substack{\bsq \in \bbZ^s \\ q_s \ne h_s}}
  \frac{\beta_{\supp(\bsh)}}{r'(\bsh)} \frac{\beta_{\supp(\bsq)}}{r'(\bsq)} e^{2\pi\ri k(\bsq-\bsh)\cdot\bsz/n},
\end{align*}
where we used the property that $(1/n)\sum_{k\in\bbZ_n} e^{2\pi\ri k
\bsell\cdot\bsz/n}$ is $1$ if $\bsell\cdot\bsz\equiv_n 0$ and is $0$
otherwise.

For each $k\in\bbZ_n$, we first ignore the condition $q_s\ne h_s$ in the
double sum over $\bsh,\bsq$ and derive
\begin{align*} 
  &\sum_{\bsh \in \bbZ^s} \sum_{\bsq \in \bbZ^s}
  \frac{\beta_{\supp(\bsh)}}{r'(\bsh)} \frac{\beta_{\supp(\bsq)}}{r'(\bsq)} e^{2\pi\ri k(\bsq-\bsh)\cdot\bsz/n}
  \,=\, \bigg(\sum_{\bsh \in \bbZ^s} \beta_{\supp(\bsh)}\frac{e^{2\pi\ri k \bsh\cdot\bsz/n}}{r'(\bsh)}\bigg)^2 \nonumber\\
  &\quad\,=\, \bigg(\sum_{\setu\subseteq\{1:s\}} \sum_{\substack{\bsh \in \bbZ^s \\ \supp(\bsh)=\setu}}
  \beta_\setu \prod_{j\in\setu} \frac{e^{2\pi\ri k h_jz_j/n}}{|h_j|^\alpha}\bigg)^2
  \,=\, \bigg(\sum_{\setu\subseteq\{1:s\}} \beta_\setu \prod_{j\in\setu} \omega(z_j,k)\bigg)^2 \nonumber\\
  &\quad\,=\, \bigg(\sum_{s\in\setu\subseteq\{1:s\}} \beta_\setu \prod_{j\in\setu} \omega(z_j,k)
  + \sum_{s\notin\setu\subseteq\{1:s\}} \beta_\setu \prod_{j\in\setu} \omega(z_j,k)\bigg)^2 \nonumber\\
  &\quad\,=\, \bigg(\omega(z_s,k)\,\sum_{\setu\subseteq\{1:s-1\}} \beta_{\setu\cup\{s\}} \prod_{j\in\setu} \omega(z_j,k)
   + \sum_{\setu\subseteq\{1:s-1\}} \beta_{\setu} \prod_{j\in\setu} \omega(z_j,k)\bigg)^2,
\end{align*}
where we noted that summing over $\bsh$ is the same as summing over
$-\bsh$ so that the double sum becomes the square of a single sum; then we
regrouped the sum according to the support of $\bsh$ and used the
definition of $\omega(z,k)$ in \eqref{eq:omega}; finally we split the sum
depending on whether or not $s$ belongs to $\setu$.

Next we need to subtract off the terms in the double sum with $q_s = h_s =
0$:
\begin{align*} 
  \sum_{\substack{\bsh \in \bbZ^s \\ h_s=0}} \sum_{\substack{\bsq \in \bbZ^s \\ q_s=0}}
  \frac{\beta_{\supp(\bsh)}}{r'(\bsh)} \frac{\beta_{\supp(\bsq)}}{r'(\bsq)} e^{2\pi\ri k(\bsq-\bsh)\cdot\bsz/n}
  &\,=\, \bigg(\sum_{\setu\subseteq\{1:s-1\}} \beta_{\setu} \prod_{j\in\setu} \omega(z_j,k)\bigg)^2,
\end{align*}
as well as the terms with $q_s = h_s \ne 0$:
\begin{align*} 
  &\sum_{\substack{\bsh \in \bbZ^s \\ h_s\ne 0}} \sum_{\substack{\bsq \in \bbZ^s \\ q_s=h_s}}
  \frac{\beta_{\supp(\bsh)}}{r'(\bsh)} \frac{\beta_{\supp(\bsq)}}{r'(\bsq)} e^{2\pi\ri k(\bsq-\bsh)\cdot\bsz/n} \nonumber\\
  &\,=\, \sum_{h_s\in\bbZ\setminus\{0\}} \frac{1}{|h_s|^{2\alpha}}
  \sum_{\bsh \in \bbZ^{s-1}} \sum_{\bsq \in \bbZ^{s-1}}
  \frac{\beta_{\supp(\bsh)\cup\{s\}}}{r'(\bsh)} \frac{\beta_{\supp(\bsq)\cup\{s\}}}{r'(\bsq)}
  e^{2\pi\ri k(\bsq-\bsh)\cdot(z_1,\ldots,z_{s-1})/n} \nonumber\\
  &\,=\, 2\zeta(2\alpha)\,\bigg(\sum_{\setu\subseteq\{1:s-1\}} \beta_{\setu\cup\{s\}} \prod_{j\in\setu} \omega(z_j,k)\bigg)^2.
\end{align*}

Combining these expressions yields
\begin{align*}
  &\theta_s(z_1,\ldots,z_s; \{\beta_\setu\}_{\setu\subseteq\{1:s\}}) \\
  &\,=\, \frac{1}{n} \sum_{k\in\bbZ_n}
  \Big( [\omega(z_s,k)]^2 - 2\zeta(2\alpha)\Big)
  \,\bigg(\sum_{\setu\subseteq\{1:s-1\}} \beta_{\setu\cup\{s\}} \prod_{j\in\setu} \omega(z_j,k)\bigg)^2 \\
  &\qquad + \frac{2}{n} \sum_{k\in\bbZ_n}
  \omega(z_s,k)\,\bigg(\sum_{\setu\subseteq\{1:s-1\}} \beta_{\setu\cup\{s\}} \prod_{j\in\setu} \omega(z_j,k)\bigg)
  \bigg(\sum_{\setu\subseteq\{1:s-1\}} \beta_{\setu} \prod_{j\in\setu} \omega(z_j,k)\bigg),
\end{align*}
which, together with \eqref{eq:Tds}, leads to the formulas in the lemma.
\end{proof}

If the quantities $V_{d,s}(k)$ and $W_{d,s}(k)$ are stored for each value
of $k \in \bbZ_n$ as $n$-vectors, denoted by $\bsv_{d,s}$ and
$\bsw_{d,s}$, respectively, then we would be able to calculate
$T_{d,s}(z_1,\ldots,z_{d-1},z_s)$ for all values of $z_s \in
\{1,\ldots,n-1\}$ at once
in terms of two matrix-vector multiplications
\begin{align*} 
  \frac{1}{n} \Psi_n \bsv_{d,s} + \frac{2}{n} \Omega_n \bsw_{d,s},
\end{align*}
with the $(n-1)\times n$ matrices
\begin{align*}
  \Omega_n &\,:=\, \big[\omega(z,k)\big]_{z\in\{1,\ldots,n-1\},\,k\in\bbZ_n},  \\
  \Psi_n &\,:=\, \big[[\omega(z,k)]^2 - 2\zeta(2\alpha)\big]_{z\in\{1,\ldots,n-1\},\,k\in\bbZ_n}.
\end{align*}
Actually the $- 2\zeta(2\alpha)$ term can be left out because it does not
affect the choice of the new component $z_s$. When $\alpha \ge 2$ is an
even integer, we can write
\begin{align*}
  \omega(z,k) \,=\, \frac{(2\pi)^\alpha}{(-1)^{\alpha/2+1}\alpha!} B_\alpha\Big(\frac{kz\bmod n}{n}\Big),
\end{align*}
where $B_\alpha$ is the Bernoulli polynomial of degree $\alpha$. Following
the standard fast CBC literature \cite{NC06a,NC06b,NC06c,Nuy14}, since the
function $\omega(z,k)$ depends only on the value of $(kz \bmod n)$, by an
appropriate reordering of the rows and columns of the matrices into a
circulant form when $n$ is prime (treating the $k=0$ column separately),
both matrix-vector multiplications can be done in $\calO(n \log(n))$
operations using FFT. For composite $n$ this is more complicated and
depends on the number of prime factors of $n$ \cite{NC06b}; we assume this
to be small and omit it in the description below.

Whether we can compute and store $V_{d,s}(k)$ and $W_{d,s}(k)$ efficiently
depends on the structure of the weights. We will investigate this for
different types of weights in the remaining sections. Our conclusion is
summarized in Theorem~\ref{thm:cost} at the end of the paper. All
construction costs are of the form
\[
  \calO(d\, n \log(n) + d\, n\, X),
\]
where $X$ reflects the cost of obtaining the values $V_{d,s}(k)$ and
$W_{d,s}(k)$ for one $k\in\bbZ_n$. As we just explained, if the
values of $V_{d,s}(k)$ and $W_{d,s}(k)$ are available we can find the best
value for $z_s$ in $\calO(n \log(n))$ operations, therefore the ``search''
cost to determine the entire generating vector is $\calO(d\, n \log(n))$.
We will store different quantities during the search in order to obtain
$V_{d,s}(k)$ and $W_{d,s}(k)$ efficiently, therefore incurring some memory
``storage'' cost. We will have to update these stored quantities in each
step after $z_s$ is chosen, thus incurring an ``update'' cost. This update
cost includes the computational complexity of recovering the values of
$V_{d,s}(k)$ and $W_{d,s}(k)$ from the stored quantities, in preparation
for the search for $z_{s+1}$. We remark that we are particularly
interested in large $d$ and large $n$ and so prefer to have linear
complexity $\calO(d\, n)$ or nearly linear complexity such as
$\calO(d\, n \log(n))$. We will show that this is possible in all cases
with respect to $n$. With respect to $d$ the complexity is $\calO(d^2
\log(d))$ for order dependent weights and POD weights, and unfortunately
it is $\calO(d^3)$ for SPOD weights.

\section{Product weights}\label{sec:product}

\begin{lemma} \label{lem:VW-prod}
In the case of product weights $\gamma_\setu = \prod_{j\in\setu}
\gamma_j$, we have for the quantities in Lemma~\ref{lem:semi}
\begin{align*}
  V_{d,s}(k)
  &\,=\, \gamma_s^2 \bigg(\prod_{j=1}^{s-1} \big(1+\gamma_j\,\omega(z_j,k)\big)^2\bigg)
  \bigg(\prod_{j=s+1}^d \big(1 + 2\zeta(2\alpha)\gamma_j^2\big)\bigg)
  \,=\, \gamma_s \, W_{d,s}(k)
  , \\
  W_{d,s}(k)
  &\,=\, \gamma_s \bigg(\prod_{j=1}^{s-1} \big(1+\gamma_j\,\omega(z_j,k)\big)^2\bigg)
  \bigg(\prod_{j=s+1}^d \big(1 + 2\zeta(2\alpha)\gamma_j^2\big)\bigg).
\end{align*}
\end{lemma}

\begin{proof}
For product weights and $\setu \cap \setw = \emptyset$ we have
\[
  \gamma_{\setu \cup \setw}
  \,=\,
  \bigg(\prod_{j \in \setu} \gamma_j\bigg)\bigg(\prod_{j \in \setw} \gamma_j\bigg)
  \,=\,
  \gamma_\setu \, \gamma_\setw
  .
\]
Therefore $W_{d,s}(k)$ from Lemma~\ref{lem:semi} simplifies to
\begin{align*}
  W_{d,s}(k)
  &= \sum_{\setw\subseteq\{s+1:d\}} \!\!\![2\zeta(2\alpha)]^{|\setw|}
  \bigg(\!\gamma_s \gamma_\setw \prod_{j=1}^{s-1} \big(1+\gamma_j\,\omega(z_j,k)\big)\!\bigg)
  \bigg(\!\gamma_\setw \prod_{j=1}^{s-1} \big(1+\gamma_j\,\omega(z_j,k)\big)\!\bigg) \\
  &= \gamma_s \bigg(\prod_{j=1}^{s-1} \big(1+\gamma_j\,\omega(z_j,k)\big)^2\bigg)
  \bigg(\prod_{j=s+1}^d \big(1 + 2\zeta(2\alpha)\gamma_j^2\big)\bigg).
\end{align*}
The simplified expression for $V_{d,s}(k)$ follows immediately.
\end{proof}

We note that the factor $\prod_{j=s+1}^d (1+2\zeta(2\alpha)\gamma_j^2)$,
appearing in both $V_{d,s}(k)$ and $W_{d,s}(k)$, does not make any
difference for the choice of the component $z_s$ and can be ignored. We
can store the $n$-vector
\begin{align*}
   P_{s-1}(k) \,:=\, \prod_{j=1}^{s-1} \big(1 + \gamma_j \, \omega(z_j, k)\big)^2,
\end{align*}
which can be updated in $\calO(n)$ operations using
\begin{align*}
   P_s(k) \,=\, \big(1 + \gamma_s \, \omega(z_s, k)\big)^2 P_{s-1}(k),
\end{align*}
starting with $P_0(k) :=1$, and overwritten in every step $s$ once the
choice of $z_s$ has been made, to be used in the search for $z_{s+1}$.

The overall cost of fast CBC construction for approximation with product
weights is $\calO(d\, n\log(n))$ operations for the search, $\calO(d\,n)$
operations for the update, and the memory requirement is $\calO(n)$. This
is consistent with the case for integration.

\section{Order dependent weights} \label{sec:order-dep}

\begin{lemma}\label{lem:VW-od1}
In the case of order dependent weights $\gamma_\setu = \Gamma_{|\setu|}$,
we have for the quantities in Lemma~\ref{lem:semi}
\begin{align*}
  V_{d,s}(k)
  &\,=\, \sum_{m=0}^{d-s} \binom{d-s}{m}\, [2\zeta(2\alpha)]^{m}\,
  \bigg(\sum_{\ell=0}^{s-1} \Gamma_{\ell+m+1}\, P_{s-1,\ell}(k)\bigg)^2, \\
  W_{d,s}(k)
  &\,=\, \sum_{m=0}^{d-s} \binom{d-s}{m}\, [2\zeta(2\alpha)]^{m}\,
  \bigg(\sum_{\ell=0}^{s-1} \Gamma_{\ell+m+1}\, P_{s-1,\ell}(k)  \bigg)
  \bigg(\sum_{\ell=0}^{s-1} \Gamma_{\ell+m}\, P_{s-1,\ell}(k) \bigg),
\end{align*}
where, with $\omega(z,k)$ defined in \eqref{eq:omega},
\begin{align} \label{eq:P-od}
  P_{s,\ell}(k) \,:=\,
  \sum_{\substack{\setu \subseteq \{1:s\} \\ |\setu| = \ell}} \prod_{j\in\setu} \omega(z_j,k)
  \qquad\mbox{for}\quad \ell = 0,\ldots,s.
\end{align}
\end{lemma}

\begin{proof}
For order dependent weights and $\setu \cap \setw = \emptyset$ we have
\[
  \gamma_{\setu \cup \setw}
  =
  \Gamma_{|\setu| + |\setw|}
  .
\]
Therefore $W_{d,s}(k)$ from Lemma~\ref{lem:semi} simplifies to
\begin{align*}
  &W_{d,s}(k)
  \,=\,
  \sum_{m=0}^{d-s}
  \sum_{\substack{\setw\subseteq\{s+1:d\} \\ |\setw|=m}}
  [2\zeta(2\alpha)]^{m}\,
  \Bigg(
  \sum_{\ell=0}^{s-1} \Gamma_{\ell+m+1} \sum_{\substack{\setu \subseteq \{1:s-1\} \\ |\setu| = \ell}} \prod_{j\in\setu} \omega(z_j,k)
  \Bigg) \\
  &\qquad\qquad\qquad\qquad\qquad\qquad\qquad\qquad\cdot
  \Bigg(
  \sum_{\ell=0}^{s-1} \Gamma_{\ell+m} \sum_{\substack{\setu \subseteq \{1:s-1\} \\ |\setu| = \ell}} \prod_{j\in\setu} \omega(z_j,k)
  \Bigg),
\end{align*}
which yields the desired formula; $V_{d,s}(k)$ is obtained analogously.
\end{proof}

Once the choice of $z_s$ has been made, the values of $P_{s,\ell}(k)$
can be updated using the recursion
\begin{align} \label{eq:P-rec-od}
  P_{s,\ell}(k)
  &\,=\,
  P_{s-1,\ell}(k)
  +
  \omega(z_s, k) \, P_{s-1,\ell-1}(k)
  ,
\end{align}
together with $P_{s,0}(k) := 1$ for all $s$ and $P_{s,\ell}(k) := 0$ for
all $\ell > s$. The vectors can be overwritten in each step $s$ if they
are updated starting from $\ell=s$ down to $\ell=1$. The storage cost is
$\calO(d\, n)$ and so is the update cost in each step.

If the values of $P_{s,\ell}(k)$ are stored, then it will require
$\calO(d^2)$ operations to compute $V_{d,s}(k)$ and $W_{d,s}(k)$ for each
$k\in\bbZ_n$ according to Lemma~\ref{lem:VW-od1}, leading to an overall
cost of $\calO(d\, n \log(n) + d^3 n)$ for the CBC construction, which is
rather high when $d$ is large and that is precisely the scenario we are
interested in. In the following lemma we derive alternative formulations
for $V_{d,s}(k)$ and $W_{d,s}(k)$ so that they can be evaluated
efficiently in $\calO(d\,\log(d))$ operations by making use of fast
matrix-vector multiplications with \emph{Hankel matrices} (i.e., constant
anti-diagonals).

\begin{lemma}\label{lem:VW-od2}
In the case of order dependent weights $\gamma_\setu = \Gamma_{|\setu|}$,
we have for the quantities in Lemma~\ref{lem:semi}
\begin{align*}
  V_{d,s}(k)
  &\,=\, \big[H_{d,s}^1 \, \bsp_{s-1}(k)\big]^\top D_{d,s} \big[H_{d,s}^1 \, \bsp_{s-1}(k)\big] , \\
  W_{d,s}(k)
  &\,=\, \big[H_{d,s}^1 \, \bsp_{s-1}(k)\big]^\top D_{d,s} \big[H_{d,s}^0 \, \bsp_{s-1}(k)\big],
\end{align*}
where, with $P_{s,\ell}(k)$ defined in \eqref{eq:P-od},
\begin{align*}
  \bsp_{s-1}(k) \,:=\,
  \big[ P_{s-1,\ell}(k) \big]_{\ell=0}^{s-1}
  \,\in\, \bbR^{s},
\end{align*}
\begin{align*}
   D_{d,s} \,:=\, \mathrm{diag}\bigg[\binom{d-s}{m} [2\zeta(2\alpha)]^{m}\bigg]_{m=0}^{d-s}
   \,\in\, \bbR^{(d-s+1)\times(d-s+1)},
\end{align*}
and $H_{d,s}^1, H_{d,s}^0 \in \bbR^{(d-s+1) \times s}$ are each a
rectangular part of a Hankel matrix:
\begin{align*}
  H_{d,s}^1 =
  \left[\!
  \begin{array}{llll}
    \Gamma_1 & \Gamma_2 & \cdots & \Gamma_s \\
    \Gamma_2 & \Gamma_3 & \cdots & \Gamma_{s+1} \\
    \;\vdots & \;\vdots & \ddots & \;\vdots \\
    \Gamma_{d-s+1} & \Gamma_{d-s+2} & \cdots & \Gamma_d \\
  \end{array} \!\!
  \right]\!,
  \quad
  H_{d,s}^0 =
  \left[\!
  \begin{array}{llll}
    \Gamma_0 & \Gamma_1 & \cdots & \Gamma_{s-1} \\
    \Gamma_1 & \Gamma_2 & \cdots & \Gamma_s \\
    \;\vdots & \;\vdots & \ddots & \;\vdots \\
    \Gamma_{d-s} & \Gamma_{d-s+1} & \cdots & \Gamma_{d-1} \\
  \end{array} \!\!
  \right]\!.
\end{align*}
\end{lemma}

\begin{proof}
Using the definition of the matrices $H^1_{d,s}$ and $H^0_{d,s}$ in the
lemma, we note that the two sums over $\ell$ from the formulas of
$V_{d,s}(k)$ and $W_{d,s}(k)$ in Lemma~\ref{lem:VW-od1} can be interpreted
as the $m$-th component of two matrix-vector products
\begin{align*}
  \sum_{\ell=0}^{s-1} \Gamma_{\ell+m+1}\,P_{s-1,\ell}(k)
  &\,=\, \big[H^1_{d,s}\,\bsp_{s-1}(k) \big]_m, \\
  \sum_{\ell=0}^{s-1} \Gamma_{\ell+m}\,P_{s-1,\ell}(k)
  &\,=\, \big[H^0_{d,s}\,\bsp_{s-1}(k) \big]_m.
\end{align*}
The outer sum over $m$ in $V_{d,s}(k)$ and $W_{d,s}(k)$ then turns the
expressions into the products involving the diagonal matrix $D_{d,s}$.
\end{proof}

Matrix-vector multiplication with a $d\times d$ Hankel matrix can be done
in $\calO(d \log(d))$ operations instead of $\calO(d^2)$ using a direct
approach. We will now elaborate on the linear algebra structure to exploit
the fast matrix-vector multiplication with our Hankel-like matrices
$H^1_{d,s}$ and $H^0_{d,s}$.

Define the $m \times m$ Hankel matrix based on the sequence $c_1, \ldots,
c_m$ to be
\[
  H(c_1,\ldots,c_m)
  \,:=\,
  \begin{bmatrix}
    c_1 & c_2 & \cdots & c_m     \\
    c_2 & c_3 & \cdots & 0 \\
    \vdots   & \vdots   & \ddots & \vdots \\
    c_m & 0 & \cdots & 0
  \end{bmatrix}
  \in \bbR^{m \times m}
  ,
\]
which is $c_m$ on the main anti-diagonal and zero below. (In general
Hankel matrices do not need to be zero under the main anti-diagonal.)
Then, our matrices $H^1_{d,s}$ are all possible submatrices of
$H(\Gamma_1,\ldots,\Gamma_d)$ spanning from the left top element (which is
$\Gamma_1$ in this case) up to an element on the main anti-diagonal (which
is $\Gamma_d$ in this case). Similarly, the matrices $H^0_{d,s}$ are
submatrices of $H(\Gamma_0,\ldots,\Gamma_{d-1})$.
For example, when $d=5$ we have
\begin{align*}
  H_{5,1}^1
  =
  \begin{bmatrix}
    \Gamma_1 \\
    \Gamma_2 \\
    \Gamma_3 \\
    \Gamma_4 \\
    \Gamma_5
  \end{bmatrix}
  , \quad
  H_{5,2}^1
  =
  \begin{bmatrix}
    \Gamma_1 & \Gamma_2 \\
    \Gamma_2 & \Gamma_3 \\
    \Gamma_3 & \Gamma_4 \\
    \Gamma_4 & \Gamma_5
  \end{bmatrix}
  , \quad
  H_{5,3}^1
  =
  \begin{bmatrix}
    \Gamma_1 & \Gamma_2 & \Gamma_3 \\
    \Gamma_2 & \Gamma_3 & \Gamma_4 \\
    \Gamma_3 & \Gamma_4 & \Gamma_5
  \end{bmatrix}
  ,
  \\
  H_{5,4}^1
  =
  \begin{bmatrix}
    \Gamma_1 & \Gamma_2 & \Gamma_3 & \Gamma_4 \\
    \Gamma_2 & \Gamma_3 & \Gamma_4 & \Gamma_5
  \end{bmatrix}
  , \quad
  H_{5,5}^1
  =
  \begin{bmatrix}
    \Gamma_1 & \Gamma_2 & \Gamma_3 & \Gamma_4 & \Gamma_5
  \end{bmatrix}
  .
\end{align*}

A matrix-vector multiplication with a general $d\times d$ Hankel matrix
can be done in $\calO(d \log(d))$, e.g., by the appropriate embedding in a
circulant matrix of size $2d \times 2d$ after reversing the rows and then
using FFTs. The cost of a matrix-vector multiplication with our
Hankel-like matrices $H_{d,s}^1$ can be bounded by $\calO(d \log(d))$ if
we consider them to be embedded in the $d\times d$ Hankel matrix
$H(\Gamma_1,\ldots,\Gamma_d)$, extend the input vector by zeros to length
$d$, apply the fast Hankel matrix-vector multiplication, and then take the
initial $d-s+1$ elements of the output vector as the result. There are of
course other ways of calculating these products. For example, for the
first and last matrices we only need $\calO(d)$ operations by a direct
calculation; and for the intermediate matrices we can find square blocks
which also have Hankel structure and do the matrix-vector multiplications
block-wise, but the matrices in the middle will then still be $\calO(d
\log(d))$. Hence we estimate the cost for all of these as $\calO(d
\log(d))$.

Using the recursion \eqref{eq:P-rec-od} and Lemma~\ref{lem:VW-od2}, we
conclude that the cost of evaluating $V_{d,s}(k)$ and $W_{d,s}(k)$ for one
$k \in \bbZ_n$ can be estimated as $\calO(d\log(d))$, and so we reduced
the total cost for the CBC algorithm to $\calO(d\, n \log(n) + d^2\log(d)
\, n)$ using $\calO(d\, n)$ memory.

If these order dependent weights have finite order $q$, i.e., $\Gamma_\ell
= 0$ for $\ell > q$, then the construction cost is $\calO(d\, n \log(n) +
d\, q \log(q) \, n)$ using $O(q\, n)$ memory.

\section{Product and order dependent (POD) weights} \label{sec:POD}

Recall that the combination of product weights and order dependent
weights is called \emph{product and order dependent} (POD) weights. We
need to modify the results from the previous two sections strategically to
get fast CBC construction for POD weights.

\begin{lemma}\label{lem:VW-pod1}
In case of POD weights $\gamma_\setu = \Gamma_{|\setu|}
\prod_{j\in\setu} \gamma_j$, we have for the quantities in
Lemma~\ref{lem:semi}
\begin{align*}
  V_{d,s}(k)
  &\,=\, \gamma_s^2 \sum_{m=0}^{d-s} C_{d,s,m}\,
  \bigg(\sum_{\ell=0}^{s-1} \Gamma_{\ell+m+1}\, P_{s-1,\ell}(k)\bigg)^2, \\
  W_{d,s}(k)
  &\,=\, \gamma_s \sum_{m=0}^{d-s} C_{d,s,m}\,
  \bigg(\sum_{\ell=0}^{s-1} \Gamma_{\ell+m+1}\, P_{s-1,\ell}(k)  \bigg)
  \bigg(\sum_{\ell=0}^{s-1} \Gamma_{\ell+m}\, P_{s-1,\ell}(k) \bigg),
\end{align*}
where, with $\omega(z,k)$ defined in \eqref{eq:omega},
\begin{align} \label{eq:C-pod}
  &C_{d,s,m} \,:=\,
  \sum_{\substack{\setw\subseteq\{s+1:d\} \\ |\setw|=m}} \prod_{j\in\setw} \big(2\zeta(2\alpha)\,\gamma_j^2\big)
  \qquad\mbox{for}\quad m=0,\ldots,d-s,
  \\
  \label{eq:P-pod}
  &P_{s,\ell}(k) \,:=\,
  \sum_{\substack{\setu \subseteq \{1:s\} \\ |\setu| = \ell}} \prod_{j\in\setu} \Big(\gamma_j\,\omega(z_j,k)\Big)
  \qquad\mbox{for}\quad \ell=0,\ldots,s,
\end{align}
\end{lemma}

\begin{proof}
For POD weights and $\setu \cap \setw = \emptyset$ we have
\[
  \gamma_{\setu \cup \setw}
  =
  \bigg(\prod_{j\in\setu} \gamma_j \bigg)
  \bigg(\prod_{j\in\setw} \gamma_j \bigg)\,
  \Gamma_{|\setu| + |\setw|}.
\]
Therefore $W_{d,s}(k)$ from Lemma~\ref{lem:semi} simplifies to
\begin{align*}
  &W_{d,s}(k)
  \,=\,
  \sum_{m=0}^{d-s}
  \sum_{\substack{\setw\subseteq\{s+1:d\} \\ |\setw|=m}}
  [2\zeta(2\alpha)]^{m}\,\bigg(\gamma_s\prod_{j\in\setw} \gamma_j^2\bigg)
  \bigg(
  \sum_{\ell=0}^{s-1} \Gamma_{\ell+m+1}
  \sum_{\substack{\setu \subseteq \{1:s-1\} \\ |\setu| = \ell}}
  \prod_{j\in\setu} \Big(\gamma_j\,\omega(z_j,k)\Big)
  \bigg) \\
  &\qquad\qquad\qquad\qquad\qquad\qquad\qquad\qquad\quad\qquad\qquad\cdot
  \bigg(
  \sum_{\ell=0}^{s-1} \Gamma_{\ell+m} \sum_{\substack{\setu \subseteq \{1:s-1\} \\ |\setu| = \ell}}
  \prod_{j\in\setu} \Big(\gamma_j\,\omega(z_j,k)\Big)
  \bigg),
\end{align*}
which yields the desired formula; $V_{d,s}(k)$ is obtained analogously.
\end{proof}

Again we obtain alternative formulations to allow efficient calculations.

\begin{lemma}\label{lem:VW-pod2}
In the case of POD weights $\gamma_\setu = \Gamma_{|\setu|}
\prod_{j\in\setu} \gamma_j$, we have for the quantities in
Lemma~\ref{lem:semi}
\begin{align*}
  V_{d,s}(k)
  &\,=\, \gamma_s^2\, \big[H_{d,s}^1 \, \bsp_{s-1}(k)\big]^\top D_{d,s}  \big[H_{d,s}^1 \, \bsp_{s-1}(k)\big] , \\
  W_{d,s}(k)
  &\,=\,  \gamma_s\, \big[H_{d,s}^1 \, \bsp_{s-1}(k)\big]^\top D_{d,s} \big[H_{d,s}^0 \, \bsp_{s-1}(k)\big],
\end{align*}
where, with $C_{d,s,m}$ defined in \eqref{eq:C-pod} and
$P_{s,\ell}(k)$ defined in \eqref{eq:P-pod},
\begin{align*}
   D_{d,s} := \mathrm{diag}\big[C_{d,s,m}\big]_{m=0}^{d-s}
   \,\in\, \bbR^{(d-s+1)\times(d-s+1)},
   \quad
  \bsp_{s-1}(k) :=
  \big[ P_{s-1,\ell}(k) \big]_{\ell=0}^{s-1}
  \,\in\, \bbR^{s},
\end{align*}
and the matrices $H_{d,s}^1, H_{d,s}^0 \in \bbR^{(d-s+1) \times s}$
are as defined in Lemma~\ref{lem:VW-od2}.
\end{lemma}

\begin{proof}
The proof is analogous to the proof of Lemma~\ref{lem:VW-od2}.
\end{proof}

The quantities $P_{s,\ell}(k)$ in \eqref{eq:P-pod} can be calculated in
essentially the same way as the case for order dependent weights in the
previous section. We now have the recursion
\begin{align} \label{eq:P-rec-pod}
  P_{s,\ell}(k)
  &\,=\,
  P_{s-1,\ell}(k)
  +
  \gamma_s\,\omega(z_s, k) \, P_{s-1,\ell-1}(k)
  ,
\end{align}
noting the extra factor $\gamma_s$ compared to \eqref{eq:P-rec-od},
together with $P_{s,0}(k) := 1$ for all $s$ and $P_{s,\ell}(k) := 0$ for
all $\ell> s$. The values can be overwritten for each step $s$ if they are
updated starting from $\ell=s$ down to $\ell=1$.

The coefficients $C_{d,s,m}$ defined in~\eqref{eq:C-pod} can also be
calculated recursively using
\begin{align} \label{eq:C-rec-pod}
  C_{d,s,m}
  &\,=\,
  C_{d,s+1,m}
  +
  2\zeta(2\alpha)\,\gamma_{s+1}^2 \, C_{d,s+1,m-1}
  ,
\end{align}
together with $C_{d,s,0} := 1$ for all $s$ and $C_{d,s,m} := 0$ for all
$m>d-s$. For each~$s$, the numbers $[C_{d,s,m}]_{m=0}^{d-s}$ can be viewed
as a vector with $d-s+1$ components. Noting that $C_{d,d,0} = 1$, the
recursion starts from the highest value of $s=d$ down to $s=1$. This can
be done at the pre-computation phase with all values stored for later use.
With varying values of $s$ and $m$, we are essentially computing and
storing a triangular matrix. This pre-computation and storage cost is
$\calO(d^2)$.

The cost to construct a $d$-dimensional generating vector $\bsz \in
\{1,\ldots,n-1\}^d$ for an $n$-point rank-$1$ lattice point set for
approximation using the CBC algorithm for POD weights is $\calO(d\, n
\log(n) + d^2 \log(d) \, n)$ using $\calO(dn)$ memory, which is the same
as the case for order dependent weights, but there is an additional
pre-computation and storage cost of $\calO(d^2)$ as indicated above.


\section{Smoothness-driven product and order dependent (SPOD) weights}
\label{sec:SPOD}

We now consider \emph{smoothness-driven product and order dependent}
(SPOD) weights of smoothness degree $\sigma \in \bbN$ of the form
\begin{align} \label{eq:spod}
  \gamma_\setu
  &=
  \sum_{\bsnu_\setu \in \{1:\sigma\}^{|\setu|}}
  \Gamma_{|\bsnu_\setu|} \prod_{j \in \setu} \gamma_{j,\nu_j},
\end{align}
where $|\bsnu_\setu| = \sum_{j\in\setu} \nu_j$. There is a sequence
$\{\gamma_{j,\nu}\}_j$ for every $\nu=1,\ldots,\sigma$. Note that for
$\setu = \emptyset$, we use the convention that the empty product is one,
and we interpret the sum over $\bsnu_\emptyset$ as a sum with a single
term $\bszero$ (or more formally the sum is over $\bsnu \in
\{0:\sigma\}^d$ with the condition that $\supp(\bsnu) = \setu$), such that
$\gamma_\emptyset = \Gamma_0$ (which in turn is typically set to $1$).

The smoothness degree $\sigma$ will most probably be related to the
smoothness parameter $\alpha$ of the function space. For example we could
have $\sigma = \alpha/2$, i.e., the number of derivatives of the
functions. We leave $\sigma$ as a general parameter below. Note that
SPOD weights with $\sigma=1$ are just POD weights.

\begin{lemma}\label{lem:VW-spod1}
In the case of SPOD weights \eqref{eq:spod}, we have for the quantities in
Lemma~\ref{lem:semi}
\begin{align*}
  V_{d,s}(k)
  &\,=\,
  \sum_{t=0}^{(d-s)\sigma} \sum_{t'=0}^{(d-s)\sigma}
  \bigg(\sum_{\ell=0}^{(s-1)\sigma} \Gamma^*_{t + \ell}\,P_{s-1,\ell}(k)\bigg)\,
  [G_{d,s}]_{t,t'}
  \bigg(\sum_{\ell'=0}^{(s-1)\sigma} \Gamma^*_{t' + \ell'}\,P_{s-1,\ell'}(k)\bigg),
  \\
  W_{d,s}(k)
  &\,=\,
  \sum_{t=0}^{(d-s)\sigma} \sum_{t'=0}^{(d-s)\sigma}
  \bigg(\sum_{\ell=0}^{(s-1)\sigma} \Gamma^*_{t + \ell}\,P_{s-1,\ell}(k)\bigg)\,
  [G_{d,s}]_{t,t'}
  \bigg(\sum_{\ell'=0}^{(s-1)\sigma} \Gamma_{t' + \ell'}\,P_{s-1,\ell'}(k)\bigg),
\end{align*}
where $\Gamma_i^* := \sum_{\nu=1}^\sigma \gamma_{s,\nu}\,\Gamma_{i+\nu}$
for $i = 0,\ldots, (d-1)\sigma$, with $\omega(z,k)$ defined in
\eqref{eq:omega},
\begin{align} \label{eq:P-spod}
  &P_{s,\ell}(k) \,:=\,
  \sum_{\substack{\bsnu \in \{0:\sigma\}^{s} \\ |\bsnu|=\ell}}
  \prod_{\substack{j=1 \\ \nu_j\ne 0}}^{s} \big(\gamma_{j,\nu_j} \,\omega(z_j,k)\big)
  \qquad\mbox{for}\quad \ell=0,\ldots,s\sigma,
 \\
 \label{eq:G-spod}
 &G_{d,s}
 \,:=\,
  \Bigg[
  \sum_{\setw\subseteq\{s+1:d\}}
  \sum_{\substack{\bsnu_\setw \in \{1:\sigma\}^{|\setw|} \\ |\bsnu_\setw|=t}}
  \sum_{\substack{\bsnu'_\setw \in \{1:\sigma\}^{|\setw|} \\ |\bsnu'_\setw|=t'}}
  \prod_{j\in\setw} \big(2\zeta(2\alpha)\,\gamma_{j,\nu_j}\,\gamma_{j,\nu_j'}\big)
  \Bigg]_{t,t'=0}^{(d-s)\sigma}.
\end{align}
\end{lemma}

\begin{proof}
For SPOD weights and $\setu \cap \setw  = \setu\cap \{s\} = \setw\cap
\{s\} = \emptyset$ we have
\begin{align*}
  \gamma_{\setu \cup \setw}
  &=
  \sum_{\bsnu_\setw \in \{1:\sigma\}^{|\setw|}}
  \bigg(\prod_{j\in\setw} \gamma_{j,\nu_j}\bigg)
  \sum_{\bsnu_\setu \in \{1:\sigma\}^{|\setu|}}
  \Gamma_{|\bsnu_\setw| + |\bsnu_\setu|}\,
  \prod_{j\in\setu} \gamma_{j,\nu_j},
  \\
  \gamma_{\setu \cup \{s\} \cup \setw}
  &=
  \sum_{\nu_s =1}^\sigma \gamma_{s,\nu_s}
  \sum_{\bsnu_\setw \in \{1:\sigma\}^{|\setw|}}
  \bigg(\prod_{j\in\setw} \gamma_{j,\nu_j}\bigg)
  \sum_{\bsnu_\setu \in \{1:\sigma\}^{|\setu|}}
  \Gamma_{|\bsnu_\setw| + |\bsnu_\setu| + \nu_s}\,
  \prod_{j\in\setu} \gamma_{j,\nu_j}.
\end{align*}
So we can write
\begin{align*}
 &\sum_{\setu\subseteq\{1:s-1\}} \gamma_{\setu\cup\setw} \prod_{j\in\setu} \omega(z_j,k) \\
 &\qquad\,=\,
 \sum_{\bsnu_\setw \in \{1:\sigma\}^{|\setw|}}
  \bigg(\prod_{j\in\setw} \gamma_{j,\nu_j}\bigg)
  \sum_{\setu\subseteq\{1:s-1\}} \sum_{\bsnu_\setu \in \{1:\sigma\}^{|\setu|}}
  \Gamma_{|\bsnu_\setw| + |\bsnu_\setu|}\,
  \prod_{j\in\setu} \big(\gamma_{j,\nu_j} \,\omega(z_j,k)\big) \\
 &\qquad\,=\,
 \sum_{\bsnu_\setw \in \{1:\sigma\}^{|\setw|}}
  \bigg(\prod_{j\in\setw} \gamma_{j,\nu_j}\bigg)
  \sum_{\bsnu \in \{0:\sigma\}^{s-1}}
  \Gamma_{|\bsnu_\setw| + |\bsnu|}\,
  \prod_{\substack{j=1 \\ \nu_j\ne 0}}^{s-1} \big(\gamma_{j,\nu_j} \,\omega(z_j,k)\big) \\
 &\qquad\,=\,
  \sum_{t=|\setw|}^{|\setw|\sigma} \sum_{\ell=0}^{(s-1)\sigma}
  \Gamma_{t + \ell}\,
  \bigg(\underbrace{\sum_{\substack{\bsnu_\setw \in \{1:\sigma\}^{|\setw|} \\ |\bsnu_\setw|=t}}
  \prod_{j\in\setw} \gamma_{j,\nu_j}}_{=:\,Q_{\setw,t}}
  \bigg)
  \bigg(\underbrace{\sum_{\substack{\bsnu \in \{0:\sigma\}^{s-1} \\ |\bsnu|=\ell}}
  \prod_{\substack{j=1 \\ \nu_j\ne 0}}^{s-1} \big(\gamma_{j,\nu_j} \,\omega(z_j,k)\big)}_{=:\, P_{s-1,\ell}(k)}
  \bigg).
\end{align*}
Similarly we obtain
\begin{align*}
 \sum_{\setu\subseteq\{1:s-1\}} \gamma_{\setu\cup\{s\}\cup\setw} \prod_{j\in\setu} \omega(z_j,k)
 &\,=\,
  \sum_{\nu_s=1}^\sigma \gamma_{s,\nu_s}
  \sum_{t=|\setw|}^{|\setw|\sigma} \sum_{\ell=0}^{(s-1)\sigma}
  \Gamma_{t + \ell+\nu_s}\,Q_{\setw,t}\,P_{s-1,\ell}(k) \\
 &\,=\,
  \sum_{t=|\setw|}^{|\setw|\sigma} \sum_{\ell=0}^{(s-1)\sigma}
  \Gamma^*_{t + \ell}\,Q_{\setw,t}\,P_{s-1,\ell}(k),
\end{align*}
where we introduced the sequence $\Gamma^*_i := \sum_{\nu=1}^\sigma
\gamma_{s,\nu}\, \Gamma_{i+\nu}$ for $i = 0,\ldots, (d-1)\sigma$.
Therefore $W_{d,s}(k)$ from Lemma~\ref{lem:semi} becomes
\begin{align*}
  W_{d,s}(k)
  &= \sum_{\setw\subseteq\{s+1:d\}} [2\zeta(2\alpha)]^{|\setw|}
  \bigg(
  \sum_{t=|\setw|}^{|\setw|\sigma} \sum_{\ell=0}^{(s-1)\sigma}
  \Gamma^*_{t + \ell}\,Q_{\setw,t}\,P_{s-1,\ell}(k)\bigg) \\
  &\qquad\qquad\qquad\qquad\qquad\qquad\qquad\quad\cdot
  \bigg(\sum_{t'=|\setw|}^{|\setw|\sigma} \sum_{\ell'=0}^{(s-1)\sigma}
  \Gamma_{t' + \ell'}\,Q_{\setw,t'}\,P_{s-1,\ell'}(k)\bigg) \\
  &=
  \sum_{t=0}^{(d-s)\sigma} \sum_{t'=0}^{(d-s)\sigma}
  \bigg(\sum_{\ell=0}^{(s-1)\sigma} \Gamma^*_{t + \ell}\,P_{s-1,\ell}(k)\bigg)\,
  [G_{d,s}]_{t,t'}
  \bigg(\sum_{\ell'=0}^{(s-1)\sigma} \Gamma_{t' + \ell'}\,P_{s-1,\ell'}(k)\bigg),
\end{align*}
where we swapped the order of summations and introduced
\begin{align*}
 [G_{d,s}]_{t,t'}
 &:=\, \sum_{\substack{\setw\subseteq\{s+1:d\} \\ |\setw| \le t\le |\setw|\sigma \\ |\setw|\le t'\le |\setw|\sigma}}
 [2\zeta(2\alpha)]^{|\setw|}\, Q_{\setw,t}\,Q_{\setw,t'} \\
 &\,=\,
  \sum_{\setw\subseteq\{s+1:d\}}
  [2\zeta(2\alpha)]^{|\setw|}\,
  \sum_{\substack{\bsnu_\setw \in \{1:\sigma\}^{|\setw|} \\ |\bsnu_\setw|=t}}
  \bigg(\prod_{j\in\setw} \gamma_{j,\nu_j}\bigg)
  \sum_{\substack{\bsnu'_\setw \in \{1:\sigma\}^{|\setw|} \\ |\bsnu'_\setw|=t'}}
  \bigg(\prod_{j\in\setw} \gamma_{j,\nu_j'}\bigg),
\end{align*}
which is equivalent to the definition \eqref{eq:G-spod}. In the equality
above we dropped the conditions $|\setw| \le t\le |\setw|\sigma$ and
$|\setw|\le t'\le |\setw|\sigma$ under the sum over $\setw$ because those
conditions are already enforced by the conditions $|\bsnu_\setw|=t$ and
$|\bsnu'_\setw|=t'$ under the sums over $\bsnu_\setw$ and $\bsnu_\setw'$.
The formula for $V_{d,s}(k)$ can be obtained analogously.
\end{proof}

The values of $P_{s,\ell}(k)$ defined by \eqref{eq:P-spod} can be computed
using the recursion
\begin{align} \label{eq:P-rec-spod}
  P_{s,\ell}(k) \,=\, P_{s-1,\ell}(k) +
  \sum_{\nu=1}^{\min(\sigma,\ell)} \gamma_{s,\nu}\,\omega(z_s,k)\, P_{s-1,\ell-\nu}(k),
\end{align}
together with $P_{s,0}(k) := 1$ for all $s$ and $P_{s,\ell}(k) := 0$ for
all $\ell > s\sigma$. The values can be overwritten for each step $s$ if
they are updated starting from $\ell=s\sigma$ down to $\ell=1$.

For each $s$, the matrix $G_{d,s}$ is a square matrix of order
$(d-s)\sigma+1$. We have the recursion which connects the elements of the
matrix $G_{d,s}$ to the elements of the smaller matrix $G_{d,s+1}$,
\begin{align} \label{eq:G-rec-spod}
 [G_{d,s}]_{t,t'}
 \,=\,
  [G_{d,s+1}]_{t,t'}
  + 2\zeta(2\alpha) \sum_{\nu=1}^{\min(\sigma,t)} \sum_{\nu'=1}^{\min(\sigma,t')} \gamma_{s+1,\nu}\,\gamma_{s+1,\nu'}\,
  [G_{d,s+1}]_{t-\nu,t'-\nu'},
 \end{align}
together with $[G_{d,s}]_{0,0} := 1$ for all $s$ and $[G_{d,s}]_{t,t'} :=
0$ for all $t > (d - s)\sigma$ or $t' > (d - s)\sigma$. Trivially, for
$s=d$ we have the $1\times1$ matrix $G_{d,d} = 1$. Similarly to the values
of $C_{d,s,m}$ in the previous section, these matrices can be computed
from the highest value $s=d$ down to $s=1$. They should be pre-computed
and all values need to be stored. The storage requirement is
$\calO(d^3\sigma^2)$ while the pre-computation cost is $\calO(d^3\sigma^4
)$ using direct calculation.

We can again formulate the expressions as matrix-vector
multiplications, but in this case we are unable to benefit from the
speed-up of Hankel matrices because the matrices $G_{d,s}$ are not
diagonal.

\begin{lemma}\label{lem:VW-spod2}
In the case of SPOD weights \eqref{eq:spod}, we have for the quantities in
Lemma~\ref{lem:semi}
\begin{align*}
  V_{d,s}(k)
  &\,=\, \big[ H_{d,s,\sigma}^*\, \bsp_{s-1}(k) \big]^\top\,
  G_{d,s} \,
  \big[ H_{d,s,\sigma}^*\, \bsp_{s-1}(k) \big] \\
  W_{d,s}(k)
  &\,=\, \big[ H_{d,s,\sigma}^*\, \bsp_{s-1}(k) \big]^\top\,
  G_{d,s} \,
  \big[ H_{d,s,\sigma}^0\, \bsp_{s-1}(k) \big],
\end{align*}
where, with $P_{s,\ell}(k)$ defined in \eqref{eq:P-spod},
\begin{align*}
  \bsp_{s-1}(k) \,:=\,
  [P_{s-1,\ell}(k)]_{\ell=0}^{(s-1)\sigma} \,\in\, \bbR^{(s-1)\sigma+1},
\end{align*}
$G_{d,s} \in \bbR^{((d-s)\sigma+1)\times ((d-s)\sigma+1)}$ is defined in
\eqref{eq:G-spod},
\begin{align*}
  H^0_{d,s,\sigma}
  \,:=\,
  \left[\!\!
  \begin{array}{llll}
    \Gamma_0 & \Gamma_1 & \cdots & \Gamma_{(s-1)\sigma} \\
    \Gamma_1 & \Gamma_2 & \cdots & \Gamma_{(s-1)\sigma+1} \\
    \;\vdots & \;\vdots & \ddots & \;\vdots \\
    \Gamma_{(d-s)\sigma} & \Gamma_{(d-s)\sigma+1} & \cdots & \Gamma_{(d-1)\sigma}
  \end{array}
  \!\!\!
  \right]\!
  \,\in\, \bbR^{((d-s)\sigma+1)\times ((s-1)\sigma+1)},
\end{align*}
and $H^*_{d,s,\sigma} \in \bbR^{((d-s)\sigma+1)\times ((s-1)\sigma+1)}$ is
defined as in $H^0_{d,s,\sigma}$ but with each entry $\Gamma_i$ in the
matrix replaced by $\Gamma^*_i := \sum_{\nu=1}^\sigma
\gamma_{s,\nu}\,\Gamma_{i+\nu}$ for $i = 0,\ldots, (d-1)\sigma$.
\end{lemma}

\begin{proof}
From Lemma~\ref{lem:VW-spod1} we can write
\begin{align*}
  W_{d,s}(k)
  &\,=\,
  \sum_{t=0}^{(d-s)\sigma} \sum_{t'=0}^{(d-s)\sigma}
  \big[ H_{d,s,\sigma}^*\, \bsp_{s-1}(k) \big]_{t}\,
  \big[ G_{d,s} \big]_{t,t'} \,
  \big[ H_{d,s,\sigma}^0\, \bsp_{s-1}(k) \big]_{t'} \\
  &\,=\,
  \big[ H_{d,s,\sigma}^*\, \bsp_{s-1}(k) \big]^\top\,
  G_{d,s} \,
  \big[ H_{d,s,\sigma}^0\, \bsp_{s-1}(k) \big].
\end{align*}
The formula for $V_{d,s}(k)$ can be obtained analogously.
\end{proof}

If the matrices $G_{d,s}$ are pre-computed and stored, the cost to
evaluate $V_{d,s}(k)$ and $W_{d,s}(k)$ for each $k\in\bbZ_n$ is
$\calO(d\,\sigma\log(d\,\sigma) + d^2\sigma^2) = \calO(d^2\sigma^2)$.
Hence, the cost to construct a $d$-dimensional generating vector $\bsz \in
\{1,\ldots,n-1\}^d$ for an $n$-point rank-$1$ lattice point set for
approximation using the CBC algorithm for SPOD weights is $\calO(d\, n
\log(n) + d^3\sigma^2 \, n)$ using $\calO(d^3\sigma^2 + d\,n)$ memory,
plus an additional pre-computation cost of $\calO(d^3\sigma^4)$.

As a consistency check, we verify that taking $\sigma = 1$ for SPOD
weights does recover our results for POD weights. Clearly the recursion
\eqref{eq:P-rec-spod} with $\sigma=1$ is precisely \eqref{eq:P-rec-pod}.
The situation with the matrices $G_{d,s}$ is slightly more complicated.
Consider first the recursion \eqref{eq:G-rec-spod} with $\sigma=1$ and
either $t=0$ or $t'=0$. Then
\begin{align*}
  [G_{d,s}]_{t,t'} \,=\, [G_{d,s+1}]_{t,t'} \,=\, \cdots \,=\, [G_{d,d}]_{t,t'}
  \,=\, \begin{cases} 1 & \mbox{if } t=t'=0, \\ 0 & \mbox{otherwise.} \end{cases}
\end{align*}
On the other hand, if $t>0$ and $t'>0$ then with $\sigma=1$ we obtain from
\eqref{eq:G-rec-spod}
\begin{align*}
 [G_{d,s}]_{t,t'}
 \,=\,
  [G_{d,s+1}]_{t,t'} + 2\zeta(2\alpha)\, \gamma_{s+1,1}^2\, [G_{d,s+1}]_{t-1,t'-1}.
\end{align*}
Taking $t=t'$, we see that the diagonal elements of the matrix $G_{d,s}$
are precisely the numbers $C_{d,s,m}$ as given by the recursion
\eqref{eq:C-rec-pod}. Taking $t\ne t'$, we see that the off-diagonal
elements in $G_{d,s}$ are obtained by combining only off-diagonal elements
from $G_{d,s+1}$; and by induction we can show that all off-diagonal
elements of all matrices are zero. This indicates that with $\sigma=1$ the
matrix $G_{d,s}$ is precisely the diagonal matrix $D_{d,s}$ in
Lemma~\ref{lem:VW-pod2}. Hence we conclude that our
Lemma~\ref{lem:VW-spod2} for SPOD weights with $\sigma=1$ is the same as
Lemma~\ref{lem:VW-pod2} for POD weights.

\section{Numerical results} \label{sec:num}

Before getting into the numerical experiments, we discuss some
equivalences between the different types of weights. First we note
trivially that the case of equal product weights $\gamma_j = a>0$ for all
$j\ge 1$ is the same as the case of order dependent weights $\Gamma_\ell =
a^\ell$ for all $\ell\ge 1$. Analogously, it is possible to re-scale POD
weights with an arbitrary parameter $a>0$ as follows
\[
  \gamma_\setu \,=\, \Gamma_{|\setu|} \prod_{j\in\setu} \gamma_j
  \,=\, \frac{\Gamma_{|\setu|}}{a^{|\setu|}} \prod_{j\in\setu} (a\gamma_j).
\]
These equivalences provide a convenient way to verify the accuracy of our
implementations for different types of weights. In scenarios where the two
sequences $\{\Gamma_\ell\}$ and $\{\gamma_j\}$ for POD weights have
drastically contradictory behaviors (e.g., $\Gamma_\ell$ grows fast with
increasing $\ell$ while $\gamma_j$ decays fast with increasing $j$), our
implementations can potentially run into numerical stability issues; we
can introduce an appropriate re-scaling parameter $a>0$ as above to
alleviate the problem.

We already mentioned that the case of SPOD weights with smoothness degree
$\sigma=1$ is precisely the case of POD weights. Additionally, if the
order dependent parts of SPOD weights are constant, $\Gamma_\ell = b>0$
for all $\ell\ge 1$, then we can write
\[
  \gamma_\setu \,=\, \sum_{\bsnu_\setu\in \{1:\sigma\}^{|\setu|}}
  \Gamma_{|\bsnu_\setu|} \prod_{j\in\setu} \gamma_{j,\nu_j}
  \,=\, b\, \prod_{j\in\setu} \underbrace{\sum_{\nu_j=1}^\sigma \gamma_{j,\nu_j}}_{=:\,\widetilde\gamma_j},
\]
that is, we have an equivalent formulation as POD weights with a constant
order dependent part, or just product weights if $b=1$.

In Figure~\ref{fig} we plot the values of $S_d(\bsz)$ against $n$ for
generating vectors~$\bsz$ constructed by the CBC algorithm based on three
different choices of weights:
\begin{enumerate}
\item Product weights: $\gamma_j = j^{-1.5\,\alpha}$;
\item POD weights: $\Gamma_\ell = \ell!/a^\ell$, $\gamma_j = a\,
    j^{-1.5\,\alpha}$;
\item SPOD weights: $\sigma = \alpha/2$, $\Gamma_\ell = \ell!/a^\ell$,
    $\gamma_{j,\nu} = a\,(2\,j^{-1.5\,\alpha})^\nu$;
\end{enumerate}
with the re-scaling parameter $a = (d!)^{1/d}$ for numerical stability. We
consider the target dimensions $d\in \{5,10,20,50,100\}$ and prime number
of points $n\in \{ 503, 1009, 2003, 4001, 8009, 16007,$ $32003, 64007,
128021\}$, and we explore two different smoothness parameters $\alpha=2$
and $\alpha=4$ to see if the theoretical rate of convergence $S_d(\bsz) =
\calO(n^{-\alpha+\delta})$, $\delta>0$, can be observed in practice. Our
weights have been chosen so that the implied constant in the big-$\calO$
bound is independent of the dimension $d$. However, the constant can still
be very large depending on the choice of weights and so the theoretical
convergence rate might not kick in until $n$ is large.

Recall that the initial $L_2$ approximation error is
$\max_{\setu\subseteq\{1:d\}} \gamma_\setu^{1/2}$, which is not the same
for different values of $d$ or different choices of weights. So it does
not make sense to directly compare the values of $S_d(\bsz)$ for different
$d$ or different weights; rather, we should compare only the rates of
convergence.

\begin{figure}\centering
 \includegraphics[width=10cm]{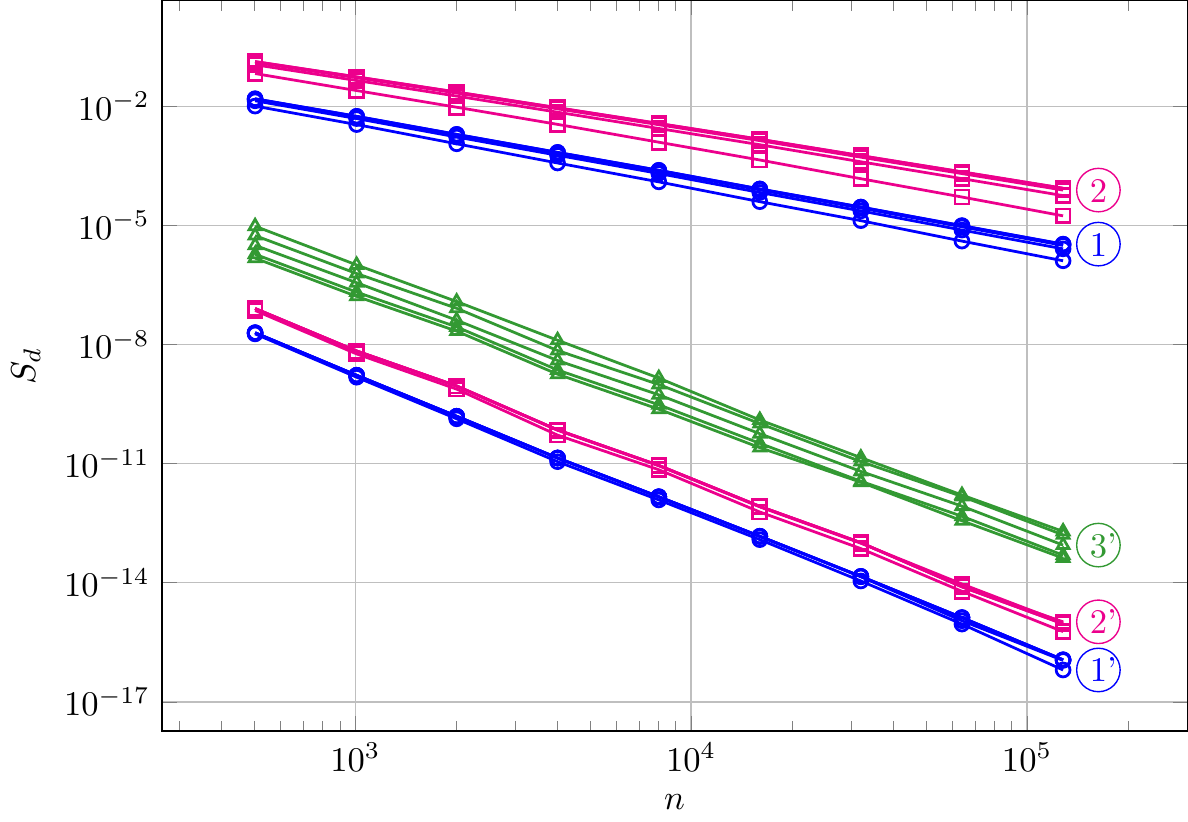}\hspace*{15mm}
 \caption{\label{fig} The values of $S_d(\bsz)$ against $n$ for different weights:
 (1) product -- blue, (2) POD -- magenta, (3) SPOD -- green,
 with $\alpha=2$ (top two groups) and $\alpha=4$ (bottom three groups).
 Each group includes five lines representing $d\in\{5,10,20,50,100\}$.
 The empirical rates of convergence for the five groups are roughly
 $n^{-1.3}, n^{-1.6}, n^{-3.1}, n^{-3.3}, n^{-3.5}$ from top down.}
\end{figure}

We see from Figure~\ref{fig} that the different values of target dimension
$d$ do not appear to affect the empirical rates of convergence, which is
consistent with our theory. For $\alpha = 2$ we observe roughly the rates
$\calO(n^{-1.3})$ for POD weights and $\calO(n^{-1.6})$ for product
weights, compared with the theoretical rate of nearly $\calO(n^{-2})$. For
$\alpha = 4$ we get roughly $\calO(n^{-3.1})$ for SPOD weights,
$\calO(n^{-3.3})$ for POD weights, and $\calO(n^{-3.5})$ for product
weights, compared with the theoretical rate of nearly $\calO(n^{-4})$.
These empirical rates exhibit the expected trend between the cases $\alpha
=2$ and $\alpha=4$.

\section{Conclusion} \label{sec:conc}

We summarize the cost of CBC construction with different forms of weights
in the theorem below.

\begin{theorem} \label{thm:cost}
The computational cost to construct a $d$-dimensional generating vector
$\bsz \in \{1,\ldots,n-1\}^d$ for an $n$-point rank-$1$ lattice point set
for approximation using the CBC construction following Algorithm~\ref{alg}
$($and satisfying Theorem~\ref{thm:final} when $n$ is prime$)$ is
\[
  \calO(d\,n\log(n) + d\,n\,X) \qquad \mbox{for search and update},
\]
where the values of $X$ for different forms of weights are summarized in
the table below, which includes pre-computation and storage costs, and a
comparison with integration.
\begin{center}
\begin{tabular}{|p{2cm}|cc|ccc|}
 \hline
         & \multicolumn{2}{c|}{Integration} & \multicolumn{3}{c|}{Approximation} \\
 \hline
 Weights & $X$ & Storage & $X$ & Pre-comp. & Storage \\
 \hline\hline
 product & $1$ & $n$ & $1$ & & $n$ \\
 \hline
 order dep. & $d$ & $d\,n$ & $d\log(d)$ & & $d\,n$ \\
 \hline
 order dep. \& finite order $q$ & $q$ & $q\,n$ & $q\log(q)$ & & $q\,n$ \\
 \hline
 POD & $d$ & $d\,n$ & $d\log(d)$ & $d^2$ & $d^2 \!+\! d\,n$ \\
 \hline
 SPOD $\sigma\!\ge\!2$ & $d\,\sigma^2$ & $d\,\sigma\,n$
      & $d^2\sigma^2$ & $d^3\sigma^4 $ & $d^3\sigma^2  \!+\! d\, n$\\
 \hline
\end{tabular}
\end{center}
\mbox{}\\
In summary, the cost is $\calO(d\,n\log(n))$ for product weights,
$\calO(d\,n\log(n) + d^2\log(d)\,n)$ for order dependent weights and POD
weights, and $\calO(d\,n\log(n) + d^3\sigma^2\,n)$ for SPOD weights with
degree $\sigma \ge 2$ $($assuming $\sigma$ is small compared to $d$ and
$n$$)$.
\end{theorem}

We see that the construction with SPOD weights is more costly than with
POD weights. When applying a lattice algorithm in an application, it may
be that the more complicated SPOD weights can lead to a better theoretical
rate of convergence when we impose the requirement that the overall error
bound is independent of dimension. There is then a potential trade-off
between the construction cost of the lattice generating vector with these
SPOD weights and the rate of convergence, which could be explored further
by the users. At the same time, we can also argue that the construction of
the generating vector is an offline cost and the user would be able to
pick an already existing generating vector, constructed for a space with
very similar SPOD weights, therefore immediately benefiting from the
better convergence rate.

The best possible rate of convergence for lattice algorithms for
approximation is proved \cite{BKUV17} to be only half of the optimal rate
of convergence for lattice rules for integration (i.e.,
$\calO(n^{-\alpha/4+\delta})$ versus $\calO(n^{-\alpha/2+\delta})$,
$\delta>0$). This is a negative point for lattice algorithms, since there
are other approximation algorithms such as Smolyak algorithms or sparse
grids which do not suffer from this loss of convergence rate. However, as
discussed in \cite{BKUV17}, lattice algorithms have their advantages in
terms of simplicity of construction and point generation, and stability
and efficiency in application, making them still attractive and
competitive despite the lower convergence rate.

Instead of measuring the worst case approximation error in the $L_2$ norm,
one can also consider other $L_p$ norms, including the $L_\infty$ norm.
Also the underlying Hilbert space $H_d$ can be changed into a Banach space
with, for example, a supremum norm. The error analysis from
\cite{CKNS-part1} as well as the fast algorithms from this paper can be
adapted.

Also related are \emph{spline algorithms} or \emph{kernel methods}
\cite{Wah90,ZLH06,ZKH09} or \emph{collocation} \cite{LH03,SNC16} based on
lattice points. In a reproducing kernel Hilbert space with a
``shift-invariant'' kernel (as we have in the periodic setting here), the
structure of the lattice points allows the required linear system to be
solved in $\calO(n\log(n))$ operations. Since splines have the smallest
worst case $L_2$ approximation error among all algorithms that make use of
the same sample points (see for example \cite{ZKH09}), the lattice
generating vectors constructed from this paper can be used in a spline
algorithm and the worst case error bound from \cite{CKNS-part1} will carry
over as an immediate upper bound with no further multiplying constant. The
advantage of a spline algorithm over the lattice algorithm \eqref{eq:Af}
is that there is no presence of the index set $\calA_d$, making it
extremely efficient in practice.

\paragraph{Acknowledgements}
We gratefully acknowledge the financial support from the Australian
Research Council (DP180101356).

\bibliographystyle{plain}

\end{document}